\documentclass[12pt]{article}
\usepackage{amsfonts}
\usepackage{amsmath}
\usepackage{amssymb}
\usepackage{amsthm}
\usepackage{setspace}
\usepackage{fullpage}
\usepackage{graphicx}
\usepackage{url}
\usepackage{verbatim}
\usepackage{enumerate}

\usepackage[pdftex,plainpages=false,hypertexnames=false,pdfpagelabels]{hyperref}
\usepackage{xcolor}
\definecolor{medium-blue}{rgb}{0,0,.8}
\hypersetup{
   colorlinks, linkcolor={black},
   citecolor={medium-blue}, urlcolor={medium-blue}
}

\usepackage{tikz}
\usetikzlibrary{arrows}

  \newcommand{\tikzmath}[2][]
     {\vcenter{\hbox{\begin{tikzpicture}[#1]#2
                     \end{tikzpicture}}}
     }

\newcommand{\nc}[2]{\newcommand{#1}{#2}}
\nc{\B}{\mathrm{B}}
\nc{\C}{\mathbb{C}}
\nc{\R}{\mathbb{R}}
\nc{\Q}{\mathbb{Q}}
\nc{\Z}{\mathbb{Z}}
\nc{\N}{\mathbb{N}}
\nc{\cA}{\mathcal{A}}
\nc{\cB}{\mathcal{B}}
\nc{\cC}{\mathcal{C}}
\nc{\cD}{\mathcal{D}}
\nc{\cI}{\mathcal{I}}
\nc{\cM}{\mathcal{M}}

\nc{\g}{\mathfrak{g}}

\def\tworarrow{\hspace{.1cm}{\setlength{\unitlength}{.50mm}\linethickness{.09mm}
                        \begin{picture}(8,8)(0,0)\qbezier(0,4)(4,7)(8,4)\qbezier(0,1)(4,-2)(8,1)\qbezier(3.5,4)(3.5,3)(3.5,1.5)
                                                 \qbezier(4.5,4)(4.5,3)(4.5,1.5)\qbezier(4,0.8)(4.5,1.7)(5.5,2)\qbezier(4,0.8)(3.5,1.7)(2.5,2)
                                                 \qbezier(8,1)(7.4,.2)(7.7,-.7)\qbezier(8,1)(7,1)(6.5,1.5)\qbezier(8,4)(7.4,4.8)(7.7,5.7)
                                                 \qbezier(8,4)(7,4)(6.5,3.5)
                        \end{picture}\hspace{.1cm}}}

\def\CS{\mathrm{CS}}

\def\Rep{\mathrm{Rep}}

\def\Bim{\mathrm{Bim}}
\def\id{\mathrm{id}}
\def\op{\mathrm{op}}

\def\Hom{\mathrm{Hom}}

\begin{document}
\title{Bicommutant categories from conformal nets}
\author{Andr\'e Henriques}
\date{}
\begin{singlespace}
\maketitle
\end{singlespace}

\newtheorem{theorem}{Theorem}[section]
\newtheorem*{theorem*}{Theorem}
\newtheorem{maintheorem}{Theorem}
\renewcommand*{\themaintheorem}{\Alph{maintheorem}}
\newtheorem*{maintheorem*}{Main Theorem}
\newtheorem{lemma}[theorem]{Lemma}
\newtheorem*{tech-lemma}{Technical lemma}
\newtheorem{proposition}[theorem]{Proposition}
\newtheorem{corollary}[theorem]{Corollary}
\newtheorem{definition}[theorem]{Definition}
\newtheorem*{definition*}{Definition}
\newtheorem{observation}[theorem]{Observation}
\newtheorem{exercise}[theorem]{Exercise}
\newtheorem{conjecture}[theorem]{Conjecture}
\newtheorem*{conjecture*}{Conjecture}
\newtheorem{question}[theorem]{Question}
\newtheorem*{question*}{Question}

\newenvironment{maindefinition}[1][Main Definition.]{\begin{trivlist}
\item[\hskip \labelsep {\bfseries #1}]}{\end{trivlist}}

\theoremstyle{remark}
\newtheorem{remark}[theorem]{Remark}
\newtheorem*{remark*}{Remark}
\newtheorem{example}[theorem]{Example}
\newtheorem*{example*}{Example}
\newtheorem*{note}{Note}

\abstract{
We prove that the category of solitons of a finite index conformal net is a bicommutant category, and that its Drinfel'd center is the category of representations of the conformal net.
In the special case of a chiral WZW conformal net with finite index, the second result specializes to the statement that the Drinfel'd center of the category of
representations of the based loop group is equivalent to the category of representations of the free loop group.
These results were announced in \cite{CS(pt)}.\bigskip
}

\tableofcontents

\newpage
\section{Introduction and statement of results}
In %our recent preprint 
\cite{CS(pt)}, we made the announcement that, at least for $G=SU(n)$, the Drinfel'd center of the category of locally normal\footnote{Local normality is a technical condition which might be equivalent to the positivity of the energy \cite[Conj.\,22 \& 34]{colimits}.}
representations of the based loop group
is equivalent, as a braided tensor category, to the category of locally normal
representations of the free loop group:
\begin{equation}\label{eq: Z(Rep^k(Omega G)) == Rep^k(LG)}
Z\big(\Rep^k(\Omega G)\big) \,\cong\, \Rep^k(LG).
\end{equation}
One of the main goals of this paper is to establish the above relation (see Theorem~\ref{Thm1} for a precise statement).

It should be noted that the representation theory of based loop groups had not been considered before.
The mere fact that the fusion product makes sense for these representations is, in itself, remarkable.

The broader relevance of the above result comes from topological quantum field theory (TQFT), specifically from Chern-Simons theory.
%As is well known, 
There are two main classes of topological quantum field theories in dimension three:
theories of Turaev-Viro type, associated to fusion categories
\cite{MR1191386, MR1357878}, and 
theories of Reshetikhin-Turaev type, associated to modular tensor categories 
\cite{MR1091619, MR1797619}
(Chern-Simons theories are of the latter kind).
Since the groundbreaking work of Jacob Lurie on the classification of extended TQFTs \cite{MR2555928}, it has been an important question to determine which theories fit into that formalism; % of extended quantum field theory;
a theory for which that is the case is said to ``extend down to points''.
It is broadly accepted (even though this has not yet been proven) that theories of Turaev-Viro type extend down to points
\cite{arXiv:1312.7188, Wray-thesis}.
On the other hand, for a typical Reshetikhin-Turaev theory, it was generally thought that this should not be possible
(the results in \cite[\S 5.5]{MR3039775} can be interpreted as a no-go theorem --- see \cite[Rem.\,5]{CS(pt)} for a discussion).

The theory of bicommutant categories (which still needs to be developed) promises to achieve two things.
First, it shows that, contrary to general expectations, Reshetikhin-Turaev theories do seem to extend down to points (at least the ones coming from conformal nets).
Second, and more importantly, it puts Turaev-Viro theories and Reshetikhin-Turaev theories on an equal footing, by providing a unified language that applies to both of them.
The expected relations are summarised in the following diagram:
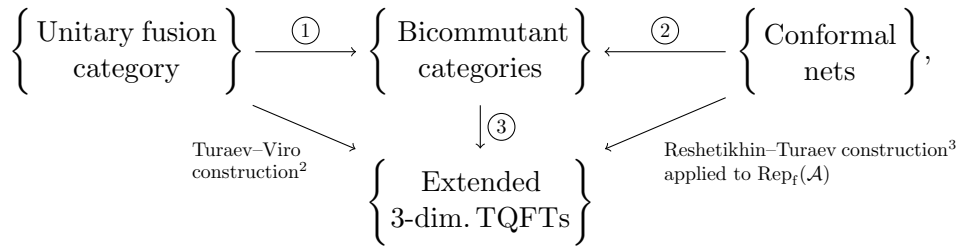
\[
\quad\,\,\,
\tikzmath{
\node(A) at (-4.7,0) {$\Bigg\{\,\parbox{2.4cm}{Unitary fusion\\\centerline{category}}\,\Bigg\}$};
\node(B) at (0,0) {$\Bigg\{\;\parbox{2.2cm}{Bicommutant \\\centerline{categories}}\;\Bigg\}$};
\node(C) at (4.7,0) {$\Bigg\{\,\parbox{1.75cm}{Conformal \\\centerline{nets}}\,\Bigg\}$,};
\node(b) at (0,-2) {$\Bigg\{\,\parbox{2.3cm}{$\,\,\;$ Extended\\\centerline{3-dim.\,TQFTs}}\,\Bigg\}$};
\draw[->] (A) --node[above, draw, circle, scale=.7, inner sep=2, yshift=4]{1} (B);
\draw[->] (C) --node[above, draw, circle, scale=.7, inner sep=2, yshift=4]{2} (B);
\draw[->] (A) --node[below, xshift=-21, yshift=-3, scale=.75]{$\parbox{2cm}{\footnotesize Turaev--Viro\\[-.04cm]\footnotesize construction\footnotemark}$} (b);
\draw[->] (C) --node[below, xshift=55, yshift=-3.5, scale=.75]{$\parbox{5.2cm}{\footnotesize Reshetikhin--Turaev construction\footnotemark\\\footnotesize applied to $\Rep_{\mathrm f}(\cA)$}$} (b);
\draw[->] (B) --node[right, draw, circle, scale=.7, inner sep=2, xshift=4]{3} (b);
}
\]
The
\addtocounter{footnote}{-1}%
\footnotetext{The Turaev-Viro construction requires the choice of a pivotal structure on the fusion category. A unitary fusion category admits a canonical pivotal structure \cite[Prop.\,8.23]{MR2183279}.}%
arrow labelled 1 was constructed in our earlier paper \cite{Bicommutant-categories-from-fusion-categories}.
The arrow labelled 2 is the content of the present paper (see Corollary~\ref{cor: main corollary} below for a precise statement).
\addtocounter{footnote}{1}%
\footnotetext{For this construction to work, one needs to assume that the conformal net $\cA$ has finite index, so that $\Rep_{\mathrm f}(\cA)$ is modular---see Remark~\ref{rem: potentially infinite direct sum}.}%
The arrow labelled 3 is still conjectural and is only expected to exist when the bicommutant category satisfies certain finiteness conditions (ensuring that it is fully dualisable).

\subsection{Motivations from Chern-Simons theory}

By the celebrated \emph{cobordism hypothesis} \cite{MR1355899, MR2555928}, a topological field theory is entirely determined by its value on a point.
The present line of research was motivated by the quest for a mathematical object that one may reasonably declare to be the value of Chern--Simons theory on a point.

Given a compact connected Lie group $G$, with classifying space $BG$, let $H^4_+(BG,\Z)$ be the subset of elements $k\in H^4(BG,\Z)$ whose image under the Chern--Weil homomorphism
\begin{equation}\label{eq: H^4(BG,Z) to Sym^2(g^*)^G}
H^4(BG,\Z)\to \mathrm{Sym}^2(\g^*)^G
\end{equation}
are positive definite metrics $\langle\cdot\,,\cdot\rangle_k$ on $\g$.
By \cite[Thm.\,6]{WZW-classification}, the map \eqref{eq: H^4(BG,Z) to Sym^2(g^*)^G} is injective and the image of $H^4_+(BG,\Z)$ under that map is, up to a scalar, the set of invariant metrics on $\g$ such that $\|X\|^2\in \Z$ for all $X$ in $\{X\in\g:\exp(X)=e\}$.

In our earlier paper \cite{WZW-classification},
given $G$ and $k\in H^4_+(BG,\Z)$ as above,
we constructed a vertex operator algebra $V_{G,k}$ and a chiral conformal net $\cA_{G,k}$,
called the chiral WZW vertex algebra and the chiral WZW conformal net, respectively.\footnote{Earlier references on these models include \cite{MR1408523}\cite{MR1822111}\cite{MR946997}\cite[\S2]{MR992362}\cite[\S6]{MR1409292}.}
A bijective correspondence was established in \cite{arXiv:1503.01260} between a certain class of unitary vertex algebras and a certain class of chiral conformal nets.
We conjecture that $V_{G,k}$ and $\cA_{G,k}$ map to each other under that correspondence, and that there is an equivalence of modular tensor categories
\[
\Rep_{\mathrm f}(V_{G,k})\,\cong\,\Rep_{\mathrm f}(\cA_{G,k}).
\]
Here, $\Rep_{\mathrm f}$ denotes the category of representations which are \emph{finite} direct sums of irreducible ones.
Assuming the above conjectures, we define $\Rep^k_{\mathrm f}(LG)$, the modular tensor category of positive energy representations of the loop group $LG$ at level $k$, to be the category
$\Rep_{\mathrm f}(V_{G,k})$, equivalently $\Rep_{\mathrm f}(\cA_{G,k})$.

Let $\CS_{G,k}$ be the Chern--Simons theory associated to the gauge group $G$ and the level $k$ \cite{MR1048699, MR990772}. 
This is a 3-dimensional topological field theory with action functional given, up to a scalar, by:\footnote{When $G$ is not simply connected,
one cannot use the formula \eqref{eq: Lagrangian of CS} to define the action.
See \cite{MR2174418, MR1048699, MR3330242} for ways to overcome this difficulty.}
\begin{equation}\label{eq: Lagrangian of CS}
\qquad S=\int\;\!\big\langle A\wedge dA\big\rangle_k+\tfrac13 \big\langle A\wedge [A\wedge A]\big\rangle_k\;\!.\!
\end{equation}
In \cite{CS(pt)},
we argued that a necessary condition for a tensor category $T$ to be the value of $\CS_{G,k}$ on a point is for its Drinfel'd center $Z(T)$ %should 
to be braided equivalent to $\Rep^k_{\mathrm f}(LG)$, or possibly $\Rep^k(LG)$
(see Section~\ref{sec: intro: Bicommutant categories} for a definition of the Drinfel'd center).
We proposed the category $\Rep^k(\Omega G)$ of \emph{locally normal representations of the based loop group} as a candidate for the value of Chern--Simons theory on a point (see \cite{MR2648901, Wray-thesis} for previous work in that direction), and offered the relation \eqref{eq: Z(Rep^k(Omega G)) == Rep^k(LG)}
as evidence for our claim.

For the remainder of this section,
let us commit to the following definitions:
\begin{equation}\label{eq: Rep^k_f = ...   Rep^k = ...   }
\Rep^k_{\mathrm f}(LG):=\Rep_{\mathrm f}(\cA_{G,k})\qquad\quad
\Rep^k(LG):=\Rep(\cA_{G,k}).
\end{equation}
%Following \cite[\S4]{CS(pt)},
Let us also define $\Rep^k(\Omega G)$
to be the category of \emph{solitons} of $\cA_{G,k}$ (see Definition~\ref{def: Solitons}, in the next section).
We call it the category of locally normal representations of the based loop group\footnote{This category is equivalent to  the version of $\Rep^k(\Omega G)$ defined in \cite[\S4]{CS(pt)} (\cite[Thm.\,31]{colimits}).}
%one $\Rep^k_{\tikzmath{\useasboundingbox (-.12,-.1) rectangle (.13,.1);\node[scale=.8]{$\scriptscriptstyle \mathrm{l.n.\!\!}$};}}(\Omega G)$, as 
%defined in \cite[\S4]{CS(pt)}.}

It is widely believed that the chiral WZW conformal nets $\cA_{G,k}$ satisfy a certain finiteness condition called \emph{finite index}, or complete rationality (see Section~\ref{sec: main results} for a definition).
This property is known to hold for $G=SU(n)$ \cite{MR1645078, MR1776984}, and in a few other cases.

\begin{theorem}\label{Thm1}
Let $\Rep^k(LG)$ be as in \eqref{eq: Rep^k_f = ...   Rep^k = ...   }.
If $\cA_{G,k}$ has finite index, then
\[
Z(\Rep^k(\Omega G)) \,\cong\, \Rep^k(LG).
\]
\end{theorem}
\begin{proof}
This is a special case of Theorem~\ref{thm: If Rep(cA_G,k) is mod...}, in Section~\ref{sec: main results}.
\end{proof}
\begin{remark}\label{rem: potentially infinite direct sum}
If a conformal net $\cA$ has finite index, then $\Rep_{\mathrm f}(\cA)$ is a modular tensor category, and $\Rep(\cA)=\mathsf{Hilb}\otimes_{\mathsf{Vec}}\Rep_{\mathrm f}(\cA)$
\cite[Cor.\,37]{MR1838752}\cite[Thm.\,3.9]{BDH-cn2}.\footnote{
The braiding on $\Rep(\cA)$ defined in \cite[Sec. 3B]{BDH-cn2} has not been compared to the one in \cite{MR1838752}.
We can therefore not exclude the possibility that, when $\mu(\cA)<\infty$, the category $\Rep(\cA)$ has two \emph{distinct} modular structures.
The braided structure used in Theorem~\ref{thm: If Rep(cA_G,k) is mod...} is the one used in \cite{MR1838752}.}
The latter implies that every object of $\Rep(\cA)$ is a (potentially infinite) direct sum of simple objects.
%In particular, every object of $\Rep(\cA)$ is a direct sum of objects of $\Rep_{\mathrm f}(\cA)$.
\end{remark}

\begin{remark}
When $G$ is not simply connected, we presently do not know, in general, whether the vertex algebra $V_{G,k}$ is unitary.
When $G\not= SU(n)$, it is not known whether $\cA_{G,k}$ is completely rational or whether $\Rep_{\mathrm f}(\cA_{G,k})$ is modular, except in some isolated cases.
Even when $G=SU(n)$, where it is known that  $\Rep_{\mathrm f}(\cA_{G,k})$ is modular, it is not known whether $\Rep_{\mathrm f}(V_{G,k})\cong\Rep_{\mathrm f}(\cA_{G,k})$ as (modular) tensor categories, except when $n=2$ \cite[\S3]{CS(pt)}.
Establishing the above properties are important open problems.
\end{remark}

\subsection{Representations and solitons}\label{sec: Representations and solitons}

Conformal nets \cite[Def.\,1.1]{BDH-cn1} are functors $\cA:\mathsf{INT}\to\mathsf{VN}$ from the category of intervals (an \emph{interval} is a manifold diffeomorphic to [0,1]) to the category of von Neumann algebras (see Definition~\ref{Def:CN} for the axioms that such a functor should satisfy).
%Let $\cA$ be a conformal net in the sense of \cite[Def.\,1.1]{BDH-cn1} (see Section~\ref{sec: co free nets} for the definition).

Let $S^1:=\{z\in\C:|z|=1\}$ be the standard circle.
A \emph{representation} of a conformal net consists of a Hilbert space $H$ and a collection of compatible actions
\[
\rho_I:\cA(I)\to \B(H)
\]
of the algebras $\cA(I)$, where $I$ ranges over all subintervals of $S^1$.
We write $\Rep(\cA)$ for the category of representations of $\cA$ whose underlying Hilbert space is separable.
(Throughout this work, all Hilbert spaces are assumed to be separable. This will be important for the results in Section~\ref{sec: The absorbing object} to hold, see Remark~\ref{rem: dim le kappa}.)

The monoidal structure on $\Rep(\cA)$ is defined as follows.
Let $H$ and $K$ be representations. 
Let $I_+$ be the upper half of $S^1$, and let $I_-$ be its lower half.
Precomposing the left action of $\cA(I_+)$ on $H$ by the map
\begin{equation}\label{eq: flip circle half}
\cA(\;\!z\mapsto\bar z\;\!:I_-\to I_+)\;:\,\,\cA(I_-)^{\op}\to \cA(I_+)
\end{equation}
yields a right action of $\cA(I_-)$ on $H$. We let
\begin{equation*} %\label{eq: H boxtimes K:=H boxtimes_cA(I_-) K}
H\boxtimes K:=H\boxtimes_{\cA(I_-)} K.
\end{equation*}
Here, the symbol $\boxtimes$ denotes Connes' relative tensor product (see Section~\ref{sec: Bim(R)} for a definition).
The algebra $\cA(I_-)$ acts on $K$ in the usual way, and it acts on $H$ on the right as described above.

The left actions of $\cA(I_-)$ on $H$ and of $\cA(I_+)$ on $K$ induce corresponding actions on $H\boxtimes K$.
For every interval $I\subset S^1$, 
the actions of\;\!\footnote{Here, we use the convention $\cA(I_1\sqcup I_2):=\cA(I_1)\,\bar\otimes\,\cA(I_2)$, where $\bar\otimes$ denotes the spatial tensor product, to define the value of $\cA$ on disjoint unions of intervals.}
$\cA(I\cap I_-)$ and $\cA(I\cap I_+)$ on $H\boxtimes K$ extend to an action
\[
\rho_I:\cA(I)\to \B(H\boxtimes K).
\]
%\cite[Def.\,1.31 and Lem.\,1.9]{BDH-cn1}
Together, these equip $H\boxtimes K$ with the structure of a representation.
We refer the reader to Section~\ref{sec: Reps of nets} for more details. % on this construction. 
There is also a braiding on $\Rep(\cA)$, discussed in Section~\ref{sec: the braiding}.

A \emph{soliton} of a conformal net is something akin to a representation \cite{MR1652746, MR1892455, MR1332979, MR2100058}
(the usage of the term `soliton' in algebraic quantum field theory goes back to at least \cite{MR0413868}):

\begin{definition}[{\cite[\S3.0.1]{MR2100058}}]\label{def: Solitons}
A soliton of $\cA$ is a Hilbert space (always assumed separable) equipped with compatible actions of the algebras $\cA(I)$, where $I$ ranges over all subintervals of the standard circle whose interior does not contain the base point $1\in S^1$.
We write $T_\cA$ for the category of solitons of $\cA$. % whose underlying Hilbert space is separable.
\end{definition}

Equivalently, a soliton is a Hilbert space equipped with compatible actions of all the algebras $\cA(I)$
as $I$ ranges over all subintervals $I\subsetneq S^1_{\mathrm{cut}}$,
where $S^1_{\mathrm{cut}}$ is the manifold obtained from the standard circle by removing its base point and replacing it by two points:
\[
S^1:\,\,\, \begin{matrix}\tikz{\draw circle (.6);\fill (.6,0) circle (.022);}\end{matrix}
\qquad\qquad
S^1_{\mathrm{cut}}:\,\,\, \begin{matrix} \tikz{\draw (0,0) arc(5:355:.6) coordinate (x);\fill (0,0) circle (.02) (x) circle (.02);}\end{matrix}
\]
Further down, we sometimes write $T^+_\cA$ in place of $T_\cA$, for reasons that will become clear later on. %\medskip

The monoidal structure on $T_\cA$ is defined in the same way as the one of $\Rep(\cA)$.
Given two solitons $H$ and $K$, 
we consider the right action $\cA(I_-)^\op\to B(H)$ given as the composite of the map \eqref{eq: flip circle half} with the left action $\cA(I_+)\to B(H)$,
and we let
\begin{equation}\label{eq: def of H boxtimes K}
H\boxtimes K:=H\boxtimes_{\cA(I_-)} K.
\end{equation}
The left actions of $\cA(I_-)$ on $H$ and of $\cA(I_+)$ on $K$ induce corresponding actions on $H\boxtimes K$.
Finally, for any interval $I\subset S^1$, $1\not\in \mathring I$,
the actions of $\cA(I\cap I_-)$ and $\cA(I\cap I_+)$ extend to an action 
\[
\rho_I:\cA(I)\to \B(H\boxtimes K).
\]
The details of his construction can be found in Section~\ref{sec: solitons as bimodules}.
%All together, these actions equip $H\boxtimes K$ with the structure of a soliton (see Section~\ref{sec: solitons as bimodules} for more details).

\begin{remark}
We remind the reader that, by definition, when $\cA=\cA_{G,k}$, the category of solitons agrees with the category $\Rep^k(\Omega G)$ of locally normal representations of the based loop at level $k$.
\end{remark}

\subsection{Bicommutant categories}\label{sec: intro: Bicommutant categories}

Bicommutant categories are higher categorical analogs of von Neumann algebras.
They are obtained by replacing the algebra $B(H)$, in the definition of a von Neumann algebra, by the tensor category $\Bim(R)$ of all bimodules over a hyperfinite factor.

Let $R$ be a hyperfinite factor, %which is not of type $\mathrm I$, 
and let $\Bim(R)$ be its category of bimodules, equipped with the monoidal structure given by Connes' relative tensor product % (see Section~\ref{sec: Bim(R)}).
(we insist that all Hilbert spaces be separable).
The category $\Bim(R)$ admits an antilinear involution at the level of objects (the conjugate of a bimodule) and a second involution at the level of morphisms (the adjoint of a linear map).
Together, these two involutions equip this category with the structure of a \emph{bi-involutive tensor category} (Definition~\ref{def: bi-involutive tensor category}). %\cite[Def.\,2.3]{Bicommutant-categories-from-fusion-categories} 

A \emph{bicommutant category} is a particular kind of bi-involutive tensor category.
% $T$ for which there exists a functor
%\[
%T\to \Bim(R)
%\]
%such that $T$ is equivalent to its bicommutant in $\Bim(R)$
%(in the cases treated in this paper, this functor will always be a fully faithful).
%We explain what this means.
Given a bi-involutive functor $\iota:T\to B$ between bi-involutive tensor categories, %\cite[Def.\,2.5]{Bicommutant-categories-from-fusion-categories},
one may consider the \emph{commutant} $Z_B(T)$ of $T$ inside $B$. %\cite[\S2.3]{Bicommutant-categories-from-fusion-categories} .
The objects of $Z_B(T)$ are pairs $(X, e)$ with $X\in B$ and $e= (e_{Y})_{Y\in T}$ a unitary half-braiding
$e_{Y}:X\otimes \iota(Y) \to \iota(Y)\otimes X$,
natural in $Y$, and subject to the `hexagon' axiom
$e_{Y_1\otimes Y_2}=(\id_{\iota(Y_1)}\otimes e_{Y_1})\circ(e_{Y_1}\otimes\id_{\iota(Y_2)})$ (see Section~\ref{sec: commutant of a tensor category} for more details).
The category $Z_B(T)$ is again bi-involutive, and is equipped with a bi-involutive functor $(X,e)\mapsto X$ to~$B$:
\[
T\,\,\rightarrow\,\, B\,\,\leftarrow\,\, Z_B(T).
\]
The \emph{Drinfel'd center} is a special case of the above notion:
\begin{definition}\label{def: Drinfel'd center}
The Drinfel'd center $Z(T)$ of a bi-involutive tensor category $T$ is the commutant of $T$ inside itself.
\end{definition}

\noindent The Drinfel'd center of a bi-involutive tensor category is braided and bi-involutive.

When $B=\Bim(R)$, we write $\cC':=Z_{\Bim(R)}(T)$ for the commutant of $T$ inside $\Bim(R)$.
There is an obvious `inclusion' functor $T\to T''$ from any category to its bicommutant
which sends an object $Y\in T$ to the object $(\iota(Y),e')$, with half-braiding $e'$ 
given by $e'_{(X,e)}:=e_{Y}^{-1}$ for $(X,e)\in \cC'$.

\begin{definition}
A bicommutant category is a bi-involutive tensor category $T$ for which there exists a hyperfinite factor $R$ and
a bi-involutive functor $T\to \Bim(R)$ such that the inclusion functor $T\to T''$ is an equivalence of (bi-involutive tensor) categories.
\end{definition}

The category of solitons of a conformal net is bi-involutive in the following way.
Given $H\in T_\cA$, with actions $\rho_I:\cA(I)\to \B(H)$ for $I\subsetneq S^1_{\mathrm{cut}}$,
its conjugate $\hspace{.2mm}\overline {\hspace{-.2mm}H}$ is the complex conjugate Hilbert space equipped with the actions
\begin{equation}\label{eq: involutive structure on T_A}
\cA(I)\xrightarrow{\cA(z\mapsto\bar z)}\cA(\bar I)^{\op}\xrightarrow{\,\,\,*\,\,\,}\overline{\cA(\bar I)}\xrightarrow{\,\,\,\overline{\rho_{\bar I}}\,\,\,}\overline{\B(H)}=\B(\overline H).
\end{equation}
Here, $\bar I$ denotes the image of $I\subset S^1$ under the complex conjugation map $S^1\to S^1$.
The conjugation operation on $T_\cA$ squares to the identity, and satisfies $\overline{H\boxtimes K}\,\cong\, \overline K\boxtimes \overline H$.

Given a conformal net $\cA$, set $R:=\cA(I_-)$. %, where $I_+$ is, as before, the upper half of the standard circle.
Then there is an obvious fully faithful bi-involutive functor
\begin{equation}\label{eq: T_A to Bim(R)}
T_\cA\to\Bim(R).
\end{equation}
It sends a soliton $H$ to the $R$-$R$-bimodule with left action given by the usual left action of $\cA(I_-)$ on $H$, and right action 
given by the left action of $\cA(I_+)$ precomposed by the map \eqref{eq: flip circle half}.
One of our main results (Corollary~\ref{cor: main corollary}) is that when $\cA$ has finite index, the above functor exhibits $T_\cA$ as a bicommutant category.

\subsection{Main results}\label{sec: main results}

%There is an important invariant $\mu(\cA)\in \R_+\cup\{\infty\}$ of a conformal net called the $\mu$-index \cite{MR1838752}.
%By definition, $\mu(\cA)$ is the Jones--Kosaki index \cite{MR696688, MR829381} of the subfactor \smallskip
%\[
%\cA(I_1)\vee \cA(I_3)\subset (\cA(I_2)\vee \cA(I_4))',\smallskip
%\]
%where $I_1, I_2, I_3, I_4\subset S^1$ are cyclically ordered intervals that touch each other at four points,
%and the prime denotes the commutant in the vacuum sector.

Recall that $T_\cA=T^+_\cA$ is the category whose objects are 
Hilbert spaces equipped with compatible actions of the algebras $\cA(I)$, for $I\subset S^1$, $1\not\in\mathring I$.

Let $T^-_\cA$ denote the category whose objects are 
Hilbert spaces equipped with compatible actions of %the algebras 
$\cA(I)$, for $I\subset S^1$, $-1\not\in\mathring I$.
Letting $R:=\cA(I_-)$,
the same formulas \eqref{eq: def of H boxtimes K} and \eqref{eq: involutive structure on T_A}
endow $T^-_\cA$ with the structure of a bi-involutive tensor category, and we have a bi-involutive functor
\[
T^-_\cA\to\Bim(R).
\]

\begin{maintheorem}\label{thm: If Rep(cA_G,k) is mod...}
Let $\cA$ be a conformal net with finite index and let $R:=\cA(I_-)$. % be the algebra it assigns to the lower semi-circle.
Let $T_\cA=T^+_\cA$ be its category of solitons, with canonical inclusion $T^+_\cA\to \Bim(R)$ as in \eqref{eq: T_A to Bim(R)}.
Then:
\begin{itemize}
\item The canonical map $(T^+_\cA)'\to \Bim(R)$ is fully faithful and we have $(T^+_\cA)'=T^-_\cA$.
\item The canonical map $(T^-_\cA)'\to \Bim(R)$ is fully faithful and we have $(T^-_\cA)'=T^+_\cA$.
\item The Drinfel'd center of $T^+_\cA$ is equivalent to $\Rep(\cA)$ as a braided bi-involutive tensor category.
\end{itemize}
\end{maintheorem}

\begin{corollary}\label{cor: main corollary}
If $\cA$ is a conformal net with finite index, then $T_{\cA}$ is a bicommutant category.
\end{corollary}

\begin{remark}
The main theorem in \cite[\S5]{CS(pt)} is stated as an equivalence of \emph{balanced} tensor categories (a balanced tensor category is a braided tensor categories with twists \cite{MR1107651}).
When $X$ is a dualizable object,
the twist $\theta_X:X\to X$ is expressible in terms of the braiding and the dagger structure as
$\theta_X:=(\mathrm{ev}_X\otimes\id)(\id\otimes \beta_{X,X})(\mathrm{ev}_X^*\otimes\id)$,
where $\mathrm{ev}_X:\overline X\otimes X\to 1$ and $\mathrm{coev}_X:1\to X\otimes \overline X$ are solutions to the normalized duality equations \cite[\S2.2]{Bicommutant-categories-from-fusion-categories}.
This can then be extended to arbitrary objects by additivity (see Remark~\ref{rem: potentially infinite direct sum}).
\end{remark}

\begin{remark}
If we do not assume that $\cA$ has finite index, then we can still define the tensor functor $T^-_\cA\to (T^+_\cA)'$ and the
braided tensor functor 
$\Rep(\cA)\to Z(T_{\cA})$,
but we do not know whether they are equivalences.
\end{remark}

\subsection*{Acknowledgments}

I am deeply indebted to Arthur Bartels and Christopher Douglas for the long-time collaboration that set the foundations on which the present work is resting,
and I am grateful to Jacob Lurie for suggesting, back in 2008, that one could use our work with A. Bartels and C. Douglas to understand what Chern-Simons theory assigns to a point.
This research was supported by the ERC grant No 674978 under the European Union's Horizon 2020 research innovation programme.

%I also thank David Ayala, Bruce Bartlett, David Ben-Zvi, Marcel Bischoff, Sebastiano Carpi, Dan Freed, Yi-Zhi Huang, Jacob Lurie, Hessel Posthuma, Chris Schommer-Pries, Urs Schreiber, Christoph Schweigert, Graeme Segal, %Noah Snyder, Peter Teichner, Constantin Teleman, and Konrad Waldorf for many fruitful discussions, help, and advice.
%Finally, many thanks to MSRI for its hospitality during the spring of 2014,\footnote{Part of this work was completed while the author was in residence at MRSI.}
%and to the Leverhulme trust for its financial support while visiting Oxford in 2014--15.

%\footnotesize
%\bibliographystyle{abbrv}        \bibliography{CS(pt)}        \end{document}

\section{Bicommutant categories}

Bicommutant categories are higher categorical analogs of von Neumann algebras.
They %notion of a bicommutant category was 
were introduced in \cite{CS(pt)}, and the first examples were constructed in \cite{Bicommutant-categories-from-fusion-categories}.

Let $R$ be a hyperfinite factor, %which is not of type $\mathrm{I}$ 
and let $R\text{-Mod}$ be the category of $R$-modules whose underlying Hilbert space is separable.
We think of $R\text{-Mod}$ as a higher categorical analog of an infinite dimensional Hilbert space.
Our slogan is:
von Neumann algebras act on Hilbert spaces;
bicommutant categories act on categories like $R\text{-Mod}$.

In this context, the higher categorical analog of $\B(H)$ is
the tensor category $\mathrm{End}(R$-Mod$)$ of completely additive endofunctors of $R\text{-Mod}$ (see \cite[Lecture 21]{Lurie(Lectures-on-vnAlg)} or \cite[\S B.VIII]{BDH-cn3} for a definition of completely additive functors).
The latter is equivalent to the tensor category $\Bim(R)$ of all $R$-$R$-bimodules.

Recall that a von Neumann algebra is an algebra which admits a map to $B(H)$ such that the natural inclusion $A\to A''$ into its bicommutant is an isomorphism:
\[
A\to B(H)\,\,\qquad\quad A= A''.
\]
Analogously, a bicommutant category is a tensor category $T$ which admits a bi-involutive functor to $\Bim(R)$ such that
the natural inclusion functor $T\to T''$ of $T$ into its bicommutant is an equivalence of categories:
\[
T\to\Bim(R)\qquad\quad T\cong T''.
\]

%plays the role of $B(H)$ in this higher categorical game. %analog of $\B(H)$.

%A von Neumann algebra is a subalgebra of $\B(H)$ which satisfies $A''=A$.
%Similarly, a bicommutant category is a subcategory of $\Bim(R)$ which satisfies $T''=T$.

\subsection{The commutant of a tensor category}\label{sec: commutant of a tensor category}

Let $T$ be a tensor category.
The Drinfel'd center $\dot Z(T)$ of $T$ is the category whose objects are pairs $(X, e)$, where $X$ is an object of $T$ and $e=(e_Y : X \otimes Y \xrightarrow{\scriptscriptstyle\cong} Y \otimes X)_{Y\in T}$ is a family of isomorphisms called a \emph{half-braiding}.
The half-braiding is required to be natural in $Y$, and to make the following diagram\footnote{Here, we have suppressed associators for brevity. By adopting this simplified notation, we do not mean to imply that our tensor categories are strict.} commute for every $Y,Z \in T$:
\begin{equation}\label{eq: ax 1/2-br}
\tikzmath{
\node(A) at (0,0) {$X \otimes Y \otimes Z$};
\node(B) at (3,1) {$Y \otimes X \otimes Z$};
\node(C) at (6,0) {$Y \otimes Z \otimes X$.\!};
\draw[->] (A) --node[above, pos=.37, xshift=-8]{$\scriptstyle e_Y\otimes \id_Z$} (B);
\draw[->] (B) --node[above, pos=.65, xshift=6]{$\scriptstyle \id_Y\otimes e_Z$} (C);
\draw[->] (A) --node[above]{$\scriptstyle e_{Y\otimes Z}$} (C);
}
\end{equation}
A morphism $(X^1, e^1)\to (X^2, e^2)$
in the Drinfel'd center is a morphism $f:X^1\to X^2$ in $T$ such that $(\id_Y\otimes f)\circ e^1_Y = e^2_Y\circ (f\otimes \id_Y)$ for every $Y\in T$.
The tensor product of two objects of $\dot Z(T)$ is given by
$(X^1, e^1)\otimes (X^2, e^2):=(X^1\otimes X^2, e^{12})$
with $e^{12}_Y:=(e^1_Y\otimes \id_{X^2})\circ(\id_{X^1}\otimes e^2_Y)$.
Finally, $\dot Z(T)$ is equipped with a braiding
\[
\beta:(X^1, e^1)\otimes (X^2, e^2)\xrightarrow{\scriptscriptstyle\cong} (X^2, e^2)\otimes (X^1, e^1)
\]
given by $e^1_{X^2}$.
Basic references include \cite{MR1107651, %A. Joyal \& R. Street: Tortile Yang-Baxter operators in tensor categories. J. Pure Appl. Alg. 71, 43-51 (1991). // 
MR1151906, %S. Majid: Representations, duals and quantum doubles of monoidal categories. Rend. Circ. Mat. Palermo Suppl. 26, 197-206 (1991). // 
MR1966525}. %From Subfactors to Categories and Topology II.

The above definition can be relativized to the case when $T$ is a subcategory of some bigger tensor category $B$
(or, more generally, when $T$ is equipped with a functor $\iota:T\to B$, not necessarily an inclusion).

\begin{definition}\label{def:  commutant of T inside B}
Let $\iota:T\to B$ be a tensor functor between tensor categories.
The commutant $\dot Z_B(T)$ of $T$ inside $B$ is the category whose objects are pairs $(X, e)$, where $X$ is an object of $B$ and
\begin{equation}\label{eq: formula half-braiding}
e=(e_Y : X \otimes \iota Y \to %\xrightarrow{\scriptscriptstyle\cong}
\iota Y \otimes X)_{Y\in T}
\end{equation}
is a collection of isomorphisms, called a half-braiding.
The half-braiding is required to be natural in $Y$, and to satisfy the following analog of \eqref{eq: ax 1/2-br} for every $Y,Z\in T$:
\begin{equation}\label{eq: ax 1/2-br ++}
\tikzmath{
\node(A) at (0,0) {$X \otimes \iota Y \otimes \iota Z$};
\node(a) at (0,-1) {$X \otimes \iota (Y \otimes Z)$};
\node(B) at (3,1) {$\iota Y \otimes X \otimes \iota Z$};
\node(C) at (6,0) {$\iota Y \otimes \iota Z \otimes X$};
\node(c) at (6,-1) {$\iota (Y \otimes Z) \otimes X$.\!\!};
\draw[->] (A) --node[above, pos=.37, xshift=-8]{$\scriptstyle e_Y\otimes \id_{\iota Z}$} (B);
\draw[->] (B) --node[above, pos=.65, xshift=6]{$\scriptstyle \id_{\iota Y}\otimes e_Z$} (C);
\draw[->] (a) --node[above]{$\scriptstyle e_{Y\otimes Z}$} (c);
\path (A) --node[rotate=90]{$\scriptstyle \cong$} (a);
\path (C) --node[rotate=90]{$\scriptstyle \cong$} (c);
}
\end{equation}
A morphism $(X^1, e^1)\to (X^2, e^2)$ in $\dot Z_B(T)$ is a morphism $f:X^1\to X^2$ in $B$ satisfying
$(\id_{\iota Y}\otimes f)\circ e^1_Y = e^2_Y\circ (f\otimes \id_{\iota Y})$ for every $Y \in B$.

The tensor product in $\dot Z_B(T)$ is given by the same formula as for the Drinfel'd center:
\[
(X^1, e^1)\otimes (X^2, e^2):=(X^1\otimes X^2, e^{12}),
\]
$e^{12}_Y:=(e^1_Y\otimes \id_{X^2})\circ(\id_{X^1}\otimes e^2_Y)$.

Finally, there is a tensor functor $\dot Z_B(T)\to B$ given by $(X,e)\mapsto X$.
\end{definition}

%\begin{remark}
\noindent
In the presence of dagger structures, the definitions of Drinfel'd center and of commutant of a tensor category inside another tensor category can be modified
by insisting that the half-braidings be \emph{unitary}.
We reserve the notations $Z(T)$ and of $Z_B(T)$ for the unitary versions.

\subsection{Bi-involutive tensor categories}

A \emph{dagger category} is a linear category over $\C$ equipped with an antilinear map $*:\Hom(X,Y)\to \Hom(Y,X)$ which satisfies $f^{**}=f$ and $(f\circ g)^*=g^*\circ f^*$.
An invertible morphism in a dagger category is called \emph{unitary} if $f^*=f^{-1}$.

\begin{definition}\label{def: tensor category}
A dagger tensor category %Selinger's survey
is a dagger category equipped with a monoidal structure whose associator and unitor isomorphisms are unitary, and which satisfies $(f\otimes g)^* =f^*\otimes g^*$.
\end{definition}

A dagger functor $F$ between dagger tensor categories is a \emph{dagger tensor functor} if %it is a dagger functor (i.e. satisfies $F(f)^*=F(f^*)$), and 
it comes along with a unitary natural transformation $\mu_{X,Y}:F(X)\otimes F(Y)\to F(X\otimes Y)$ and a unitary $i:1\to F(1)$ such that %the following identities hold:
$\mu_{X,Y\otimes Z}\circ(\id_{F(X)}\otimes\mu_{Y,Z}) = \mu_{X\otimes Y, Z}\circ (\mu_{X,Y}\otimes \id_{F(Z)})$ and
$\mu_{1,X}\circ (i\otimes \id_{F(X)}) = \id_{F(X)} = \mu_{X,1}\circ(\id_{F(X)}\otimes i)$. 

A dagger functor between dagger tensor categories is a \emph{dagger anti-tensor functor} if
it comes with a unitary natural transformation $\nu_{X,Y}:F(X)\otimes F(Y)\to F(Y\otimes X)$ and a unitary $j:1\to F(1)$ such that
$\nu_{X,Z\otimes Y}\circ (\id_{F(X)}\otimes\nu_{Y,Z}) = \nu_{Y\otimes X, Z}\circ (\nu_{X,Y}\otimes\id_{F(Z)})$ and
$\nu_{1,X}\circ(j\otimes \id_{F(X)}) = \id_{F(X)} = \nu_{X,1}\circ(\id_{F(X)}\otimes j)$.

Bi-involutive tensor categories are dagger tensor categories equipped with a second involution, denoted $X\mapsto \overline X$,
which is a dagger anti-tensor functor:

\begin{definition}\label{def: bi-involutive tensor category}
A bi-involutive tensor category is a dagger tensor category $T$ equipped with a covariant anti-linear dagger anti-tensor functor
\[
\overline{\,\cdot\,}:T\to T
\]
called the conjugate.
The structure data of this anti-tensor functor are denoted
\[
\nu_{X,Y}:\overline{X} \otimes \overline{Y} \stackrel\simeq\longrightarrow \overline{Y \otimes X}\qquad\text{and}\qquad j:1\to\overline 1.
\]
This functor is involutive, meaning that for every $X\in T$, we are given unitary natural isomorphisms $\varphi_X:X\to \overline{\overline{X}}$ satisfying $\varphi_{\overline{X}}=\overline{\varphi_{X}}$.
Finally, we require the compatibility conditions $\varphi_1=\overline j\circ j$ and $\varphi_{X \otimes Y}=\overline{\nu_{Y,X}}\circ\nu_{\overline X,\overline Y}\circ(\varphi_X\otimes\varphi_Y)$.
\end{definition}

\begin{definition} A dagger tensor functor $F$ between bi-involutive tensor categories is called a bi-involutive functor if 
it comes equipped with a unitary natural transformation
\[
\gamma_X:F(\overline X)\to \overline{F(X)}
\]
satisfying
$\gamma_{\overline X}=
\overline{\gamma_{X}}^{-1}\circ\varphi_{F(X)}\circ F(\varphi_X)^{-1}
$,\,
$\gamma_{1}=
\overline i\circ j\circ i^{-1}\circ F(j)^{-1}$, and
$\gamma_{X\otimes Y}=
\overline{\mu_{X,Y}}\circ\nu_{F(Y),F(X)}\circ (\gamma_Y\otimes\gamma_X)\circ\mu_{\overline Y,\overline X}^{-1}\circ F(\nu_{Y,X})^{-1}$.
\end{definition}

The prototypical example of a bi-involutive category is the category $\Bim(R)$ of all bimodules over a von Neumann algebra $R$.

Given a bi-involutive functor $T\to B$ between bi-involutive tensor categories, we let $Z_B(T)$ be the full subcategory of %$\dot Z_B(T)$
the category described in Definition~\ref{def:  commutant of T inside B} where the half-braidings \eqref{eq: formula half-braiding} are unitary.
This category has the advantage of being, once again, a bi-involutive tensor category.
The dagger structure is inherited from $B$, and the conjugate of an object $(X,e)\in Z_B(T)$ is given by $(\overline X, e')$, with
\[
e'_{Y}:
\overline X \otimes Y
\xrightarrow{\id \otimes \varphi_Y}
\overline X \otimes\overline{\overline{Y}}
\xrightarrow{\nu_{X,\overline{Y}}}
\overline{\overline{Y} \otimes X}
\xrightarrow{\overline{e_{\overline{Y}}}^{-1}}
\overline {X \otimes \overline{Y}}
\xrightarrow{\nu_{\overline Y,X}^{-1}}
\overline{\overline{Y}}\otimes \overline X
\xrightarrow{\varphi_Y^{-1}\otimes \id}
Y\otimes \overline X.
\]
The Drinfel'd center $Z(T)$ of a bi-involutive tensor category $T$ is the commutant of $T$ inside itself.

\begin{remark}
The categories $Z_B(T)$ and $\dot Z_B(T)$ need not, in general, be equivalent.
However, in the cases studied in this paper, they will turn out %to be
equivalent (see Remark~\ref{rem: Z versus dot Z}).
\end{remark}

%A \emph{bimodule} between von Neumann algebras $A$ and $B$ is a Hilbert space $H$ equipped with commuting actions $A\to \B(H)$ and $B^\op\to \B(H)$ (all actions are assumed to be continuous for the ultraweak topology).
%The collection of all von Neumann algebras and all their bimodules forms a bicategory \cite{MR1867561},
%and $\Bim(R)$ is the tensor category of endomorphisms of the object $R$.
%a certain $A$-$A$-bimodule denoted $L^2(A)$ and called the \emph{standard form}, or non-commutative $L^2$-space of $A$
%\cite{MR0407615, MR1203761}.

\subsection{Bim(\emph{R})}\label{sec: Bim(R)}

Let $R$ be a hyperfinite factor, %which is not of type $\mathrm{I}$,
and let $\Bim(R)$ be the category of all $R$-$R$-bimodules whose underlying Hilbert spaces is separable.
It is a dagger category by means of the operation that sends a bimodule map to its adjoint.

Let $L^2(R)$ be the \emph{non-commutative $L^2$-space} of $R$ \cite{MR0407615, MR1203761} (for any faithful state $\phi:R\to\C$, there is a canonical identification between $L^2R$ and the GNS Hilbert space associated to $\phi$ \cite[Def IX.1.18]{MR1943006}).
By Tomita-Takesaki theory, this Hilbert space is equipped with two actions of $R$ that are each other's commutants, and an antilinear involution $J$ that satisfies $J(x\xi y)=y^*J(\xi)x^*$.

The tensor structure
\begin{equation}\label{eq: CF}
\boxtimes_R:\,\Bim(R)\times \Bim(R)\to \Bim(R).
\end{equation}
on $\Bim(R)$ is known as \emph{Connes fusion}, or relative tensor product. % \cite{MR1303779, MR703809}.
The bimodule $L^2(R)$ is the unit object for that operation.

Given two bimodules $H$ and $K$, their fusion $H\boxtimes_RK$ is the completion of
$\Hom_{\mathrm{Mod}\text{-}R}($ $L^2R,H)\otimes_R K$
with respect to the inner product
$\langle\varphi\otimes \xi,\psi\otimes \eta\rangle:=\langle \lambda^{-1}(\psi^*\circ \varphi)(\xi),\eta\rangle$, where 
$\lambda:R\to \mathrm{End}_{\mathrm{Mod}\text{-}R}(L^2R)$ denotes the left action of $R$ on its $L^2$-space.
The left action of $R$ on $H$ and the right action of $R$ on $K$ equip $H\boxtimes_RK$, once again, with the structure of a bimodule.
The Connes fusion can be equivalently described as a completion of 
$H \otimes_R \Hom_{R\text{-}\mathrm{Mod}}(L^2R,K)$,
or
$\Hom_{\mathrm{Mod}\text{-}R}(L^2R,H)\otimes_R L^2R \otimes_R \Hom_{R\text{-}\mathrm{Mod}}(L^2R,K)$.
Basic references include \cite{BDH(Dualizability+Index-of-subfactors)}\;\!\cite[V.B.$\delta$]{MR1303779}\;\!\cite{MR703809}\;\!\cite{MR2771095}.

\begin{remark}\label{rm: HK = KH}
Using the last description of the fusion, and the fact that $L^2(R) \cong L^2(R^\op)$, we see that there is a canonical isomorphism $H\boxtimes_R K \cong K\boxtimes_{R^\op} H$.
\end{remark}

Given $H\in\Bim(R)$, its complex conjugate $\overline H$ is a bimodule by means of the actions $a\bar\xi b:=\overline{b^*\xi a^*}$.
We call it the \emph{conjugate} bimodule.
This operation comes with canonical isomorphisms
\[
\nu:\overline H\boxtimes_R \overline K\to \overline {K\boxtimes_R H}\qquad\text{and}\qquad j: L^2(R)\to \overline{L^2(R)}
\]
reviewed in \cite[\S2.4]{Bicommutant-categories-from-fusion-categories}.

All together, these operations endow $\Bim(R)$ with the structure of a bi-involutive tensor category.
By definition, a \emph{bicommutant category} is a bi-involutive tensor category $T$ for which there exists a hyperfinite factor $R$ and
a bi-involutive functor $T\to \Bim(R)$ such that the inclusion functor $T\to T'' = Z_{\Bim(R)}(Z_{\Bim(R)}(T))$ is an equivalence of categories.

\section{Conformal nets}

In this section, we recall the definition of conformal net from \cite{BDH-cn1}, along with the notion of representation of a conformal net,
the fusion product
\[
\boxtimes:\Rep(\cA)\times \Rep(\cA)\to \Rep(\cA),
\]
and the braiding of representations $\beta_{H,K}:H\boxtimes K\to K\boxtimes H$.

\subsection{Coordinate free conformal nets}\label{sec: co free nets}

Let us define an \emph{interval} to be an oriented manifold diffeomorphic to $[0,1]$.
We write $\mathrm{Diff}_+(I)$ for the group of orientation preserving diffeomorphisms of an interval $I$.
Let $\mathsf{INT}$ be the category whose objects are intervals and whose morphisms are embeddings, not necessarily orientation preserving,
and let $\mathsf{VN}$ be the category whose objects are von Neumann algebras and whose morphisms are normal maps which are either $*$-algebra homomorphisms or $*$-algebra anti-homomorphisms.

\begin{definition}[{\cite[Def.\,1.1]{BDH-cn1}}] \label{Def:CN}\rm
A conformal net is a covariant functor $\cA \colon \mathsf{INT} \to \mathsf{VN}$
from the category of oriented intervals and embeddings to the category of von Neumann algebras.
It sends 
orientation-preserving embeddings to injective homomorphisms and 
orientation-reversing embeddings to injective antihomomorphisms.
Moreover, for any intervals $I$ and $J$, 
the natural map $\Hom_{\mathsf{INT}}(I,J)\to \Hom_{\mathsf{VN}}(\cA(I),\cA(J))$ should be continuous 
for the $\mathcal C^\infty$ topology on $\Hom_{\mathsf{INT}}(I,J)$ and 
Haagerup's $u$-topology\footnote{Topology of pointwise convergence on the preduals.} on $\Hom_{\mathsf{VN}}(\cA(I),\cA(J))$.
In addition to that, a conformal net should satisfy the following five axioms:
\begin{enumerate}[\it i.]
\item \emph{Locality:} If $I,J\subset K$ have disjoint interiors, then $\cA(I)$ and $\cA(J)$ are commuting subalgebras of $\cA(K)$.
\item \emph{Strong additivity:} If $K = I \cup J$, then $\cA(K)$ is generated as a von Neumann algebra by its two subalgebras: $\cA(K) = \cA(I) \vee \cA(J)$.
\item \emph{Split property:} If $I,J\subset K$ are disjoint, then the natural map from the algebraic tensor product $\cA(I) \otimes_{\mathrm{alg}} \cA(J) \to \cA(K)$ extends to a map from the spatial tensor product $\cA(I) \, \bar{\otimes} \, \cA(J) \to \cA(K)$.
\item \emph{Inner covariance:} If $\varphi\in\mathrm{Diff}_+(I)$ restricts to the identity in a neighbourhood of the boundary of $I$, then $\cA(\varphi):\cA(I)\to \cA(I)$ is an inner automorphism. 
\item \emph{Vacuum sector:} 
Let $J \subsetneq I$ contain the boundary point $p \in \partial I$, and let $\bar{J}$ denote $J$ with reversed orientation; $\cA(J)$ acts on $L^2(\cA(I))$ via the left action of $\cA(I)$, and $\cA(\bar{J}) \cong \cA(J)^\op$ acts on $L^2(\cA(I))$ via the right action of $\cA(I)$.
In that case, we require that the action of $\cA(J) \otimes_{\mathrm{alg}} \cA( \bar{J} )$ on $L^2(\cA(I))$ extends to an action of $\cA(J \cup_p \bar{J})$\footnote{Here, $J \cup_p \bar{J}$ is equipped with any smooth structure extending the given smooth structures on $J$ and $\bar J$, and for which the orientation-reversing involution that exchanges $J$ and $\bar{J}$ is smooth.}:
\begin{equation}\label{eq: Vacuum sector axiom for nets}
\qquad\tikzmath{
\node (a) at (0,1.2) {$\cA(J) \otimes_{\mathrm{alg}} \cA( \bar{J} )$};
\node (b) at (3.5,1.2) {$\B(L^2\cA(I))$}; 
\node (c) at (0,0) {$\cA(J \cup_p \bar{J})$}; 
\draw[->] (a) -- (b);
\draw[->] (a) -- (c);
\draw[->,dashed] (c) -- (b);
}
\end{equation}
\end{enumerate}
\end{definition}

\noindent
A conformal net $\cA$ is called \emph{irreducible} if the algebras $\cA(I)$ are factors.
We will always assume that our conformal nets are irreducible.

\begin{remark}\label{rm: or rev}
For any interval $I$, the identity map $\bar I\to I$ (which is orientation reversing) induces an isomorphism $\cA(\bar I)\cong\cA(I)^\op$.
This was used above, in the formulation of the vacuum sector axiom.
\end{remark}

Let $S^1:=\{z\in \C:|z|=1\}$.
A representation of $\cA$ consists of a Hilbert space $H$ (always assumed separable) and a collection of homomorphisms
$\rho_I \colon \cA(I) \to \B(H)$
for every interval $I\subset S^1$,
subject to the compatibility condition $\rho_I|_{\cA(J)}=\rho_J$ whenever $J \subset I$.

The vacuum representation, or vacuum sector, is a representation of $\cA$ on the Hilbert space
\begin{equation*} %\label{eq: vac sec}
H_0=H_0^\cA:=L^2(\cA(\tikz[scale=.5, rotate=180]{\useasboundingbox (-.5,-.2) rectangle (.5,.2);\draw (-.5,0) arc(180:0:.5);\draw[->] (-.05,.5) -- +(-.01,0);})).
\end{equation*}
The algebra 
$\cA(\tikz[scale=.5, rotate=180]{\useasboundingbox (-.5,-.2) rectangle (.5,.2);\draw (-.5,0) arc(180:0:.5);\draw[->] (-.05,.5) -- +(-.01,0);})$ acts via the usual left multiplication on its $L^2$-space.
The algebra
$\cA(\tikz[scale=.5, rotate=180]{\useasboundingbox (-.5,-.2) rectangle (.5,.2);\draw (-.5,0) arc(-180:0:.5);\draw[->] (.05,-.5) -- +(.01,0);})$
acts via the isomorphism 
$\cA(b):\cA(\tikz[scale=.5, rotate=180]{\useasboundingbox (-.5,-.2) rectangle (.5,.2);\draw (-.5,0) arc(-180:0:.5);\draw[->] (.05,-.5) -- +(.01,0);}) \to
\cA(\tikz[scale=.5, rotate=180]{\useasboundingbox (-.5,-.2) rectangle (.5,.2);\draw (-.5,0) arc(180:0:.5);\draw[->] (-.05,.5) -- +(-.01,0);})^\op$, where $b(z)=\bar z$,
followed by right multiplication of
$\cA(\tikz[scale=.5, rotate=180]{\useasboundingbox (-.5,-.2) rectangle (.5,.2);\draw (-.5,0) arc(180:0:.5);\draw[->] (-.05,.5) -- +(-.01,0);})$ on its $L^2$-space.
The vacuum sector axiom then ensures that those two actions uniquely extend to actions of $\cA(I)$ for every $I\subset S^1$~\cite[\S1.\sc b]{BDH-cn1}.

Given an interval $I\subset S^1$,
let $I':=S^1{\setminus} \mathring I$ denote the complement of the interior of $I$.
The representation $H_0$ satisfies the important property of Haag duality \cite[Prop~1.18]{BDH-cn1}:
\begin{equation*} %\label{eq:Haag}
\rho_{I'}(\cA(I'))=\rho_I(\cA(I))'.
\end{equation*}
We note that Definition \ref{Def:CN} is rigged in such a way so as to have Haag duality essentially built into it.
Using the classical definition of conformal nets, Haag duality is an important theorem
\cite[\S II.2]{MR1231644}\cite[Thm~6.2.3]{Longo(Lectures-on-Nets-II)}\cite{MR1079298}.

Recall that $PSU(1,1)$ is the group of M\"obius transformations, acting on $S^1$ by $(\begin{smallmatrix}a&b\\\bar b&\bar a\end{smallmatrix}):z\mapsto \frac{a z+b}{\bar b z+\bar a}$ for $|a|^2-|b|^2=1$.
Then the vacuum sector admits a continuous representation
\begin{equation}\label{eq: Su11 action}
\begin{split}
PSU(1,1)&\to U(H_0)\\
\varphi\,&\mapsto\, u_\varphi
\end{split}
\end{equation}
which satisfies the covariance property $\cA(\varphi)(a)=u_\varphi a u_\varphi^*$ for every $I\subset S^1$  and $a\in \cA(I)$ \cite[Thm 2.13]{BDH-cn1}.
%(Note that in the classical definition of a conformal nets \cite{MR1231644, Longo(Lectures-on-Nets-II)}, the action \eqref{eq: Su11 action} is simply part of the data.)
Let $r_t=(\!\!\begin{smallmatrix} e^{it\hspace{-.2mm}/\hspace{-.1mm}2}\,\,\,0\,\,\, \\ \,\,\,\,0\,\,\,\, e^{-it\hspace{-.2mm}/\hspace{-.1mm}2}\end{smallmatrix}\!)\in PSU(1,1)$ denote rotation by $t$, and let $R_t=u_{r_t}$ be its image under the above homomorphism.
The unbounded self-adjoint operator $L_0:=-i\tfrac{d}{dt}\big|_{t=0}\,R_t$ is called the \emph{energy operator};
it generates the subgroup of rotations in the sense that~$R_t=e^{tiL_0}$.
We call a conformal net \emph{chiral} if the energy operator $L_0$ has positive spectrum
and the $PSU(1,1)$-invariant subspace $H_0$ is one dimensional (equivalently, the subspace invariant under all the $R_t$'s is one dimensional).

\begin{remark}
Our definition of conformal net (Definition~\ref{Def:CN}) is different from the one usually encountered in the literature.
If a conformal net in the sense of \cite{MR1231644, Longo(Lectures-on-Nets-II)} satisfies the additional assumptions
of strong additivity and diffeomorphism covariance\footnote{The split property was recently shown to be a consequence of diffeomorphism covariance
\cite{arXiv:1609.02169}.}, then it induces a conformal net in the sense of Definition~\ref{Def:CN} \cite[Prop~4.9]{BDH-cn1}.
Conversely, a conformal net (in the sense of Definition~\ref{Def:CN}) which is chiral induces a conformal net in the classical sense by restricting it to the circle.
This establishes a bijective correspondence between chiral conformal nets in the sense described above, and 
conformal nets in the sense of \cite{MR1231644, Longo(Lectures-on-Nets-II)} subject to the additional assumptions
of strong additivity and diffeomorphism covariance.
\end{remark}

\subsection{Fusion of representations}\label{sec: Reps of nets}

The standard monoidal structure %$\boxtimes:\,\Rep(\cA)\times \Rep(\cA)\to \Rep(\cA)$
\begin{equation} \label{eq: fusion product}
\boxtimes:\,\Rep(\cA)\times \Rep(\cA)\to \Rep(\cA)
\end{equation}
on the category of representations of a conformal net is called \emph{fusion}.
There are two main approaches for defining that operation \cite[Sec. IV.2]{MR1231644} and \cite[Sec.~30]{MR1645078}
(see \cite[Prop V.B.$\delta$.17]{MR1303779} for the equivalence between the two).
We will follow the latter.

let $I_0$ be a fixed `standard interval' (in Section~\ref{sec: Representations and solitons}, this was taken to be the lower half of $S^1$, but $I_0=[0,1]$ is also a pleasant choice),
and let $R:=\cA(I_0)$ be the value of our conformal net on that standard interval.
A representation $H\in\Rep(\cA)$ has commuting left actions of the algebras
\begin{equation}\label{eq: identifications}
\tikzmath{
\node at (-.12,-.2) {$
\cA(\tikz[scale=.5]{\useasboundingbox (-.5,-.2) rectangle (.5,.2);\draw (-.5,0) arc(-180:0:.5);\draw[->] (.05,-.5) -- +(.01,0);})\cong R\,\,
\quad\text{and}\quad\,
\cA(\tikz[scale=.5]{\useasboundingbox (-.5,-.2) rectangle (.5,.2);\draw (-.5,0) arc(180:0:.5);\draw[->] (-.05,.5) -- +(-.01,0);})\cong R^\op$,};
\node[scale=.8] (X) at (-4.81,-1) {\parbox{3.8cm}{\small \centerline{constant-speed\,\,} \centerline{orientation preserving\,\:} parametrization by $I_0$}};
\node[scale=.8] (Y) at (4.66,-1) {\parbox{3.8cm}{\small \centerline{constant-speed\,\,} \centerline{orientation reversing\:\;} parametrization by $I_0$}};
\draw[->, rounded corners=4] (X.east) -- ++(1.5,0) -- +(0,.5);
\draw[->, rounded corners=4] (Y.west) -- ++(-1.21,0) -- +(0,.5);
}
\end{equation}
so we get commuting left and right actions of $R$, making $H$ into an $R$-$R$-bimodule.
The above construction gives a functor
$\Rep(\cA)\rightarrow \Bim(R)$.
We will see later, in Section~\ref{sec: solitons as bimodules},
that this functor is fully faithful (Lemma \ref{lem: T --> Bim(R) is fully faithful}), and that its image is closed under the operation \eqref{eq: CF} of Connes fusion (Lemma~\ref{lem: T_A and Rep(A) closed under fusion}).
This will allow us to define the fusion product on $\Rep(\cA)$ as the restriction of the corresponding operation on $\Bim(R)$.
%We will equip $\Rep(\cA)$ with a tensor structure by constructing a fully faithful functor
%\[
%\Rep(\cA)\rightarrow \Bim(R)
%\]
%into the category of all $R$-$R$-bimodules, and proving that its essential image is closed under Connes fusion.
%To construct the above functor, we note that a representation $H$ has commuting left actions of the algebras
%\begin{equation}\label{eq: identifications}
%\tikzmath{
%\node at (-.12,-.2) {$
%\cA(\tikz[scale=.5]{\useasboundingbox (-.5,-.2) rectangle (.5,.2);\draw (-.5,0) arc(-180:0:.5);\draw[->] (.05,-.5) -- +(.01,0);})\cong R\,\,
%\quad\text{and}\quad\,
%\cA(\tikz[scale=.5]{\useasboundingbox (-.5,-.2) rectangle (.5,.2);\draw (-.5,0) arc(180:0:.5);\draw[->] (-.05,.5) -- +(-.01,0);})\cong R^\op$,};
%\node[scale=.8] (X) at (-4.81,-1) {\parbox{3.8cm}{\small \centerline{constant-speed\,\,} \centerline{orientation preserving\,\:} parametrization by $I_0$}};
%\node[scale=.8] (Y) at (4.66,-1) {\parbox{3.8cm}{\small \centerline{constant-speed\,\,} \centerline{orientation reversing\:\;} parametrization by $I_0$}};
%\draw[->, rounded corners=4] (X.east) -- ++(1.5,0) -- +(0,.5);
%\draw[->, rounded corners=4] (Y.west) -- ++(-1.21,0) -- +(0,.5);
%}
%\end{equation}
%and so we get commuting left and right actions of $R$, making $H$ into an $R$-$R$-bimodule.
%We will equip $\Rep(\cA)$ with a tensor structure by constructing a fully faithful functor

\begin{remark}
In our situation of interest, there is an alternative way of defining Connes fusion that avoids $L^2$-spaces and Tomita-Takesaki theory,
and avoids the parametrizations \eqref{eq: identifications}.
It is defined directly as an operation on the category of
$\cA(\tikz[scale=.5]{\useasboundingbox (-.5,-.2) rectangle (.5,.2);\draw (-.5,0) arc(-180:0:.5);\draw[->] (.05,-.5) -- +(.01,0);})$%
-%
$\cA(\tikz[scale=.5]{\useasboundingbox (-.5,-.2) rectangle (.5,.2);\draw (-.5,0) arc(180:0:.5);\draw[->] (-.05,.5) -- +(-.01,0);})^\op$%
-bimodules,
equivalently,
$\cA(\tikz[scale=.5]{\useasboundingbox (-.5,-.2) rectangle (.5,.2);\draw (-.5,0) arc(-180:0:.5);\draw[->] (.05,-.5) -- +(.01,0);})$%
-%
$\cA(\tikz[scale=.5]{\useasboundingbox (-.5,-.2) rectangle (.5,.2);\draw (-.5,0) arc(180:0:.5);\draw[->] (.05,.5) -- +(.01,0);})$%
-bimodules.
All that we use is the fact that there is a distinguished bimodule $H_0$ with the property that the actions of 
$A:=\cA(\tikz[scale=.5]{\useasboundingbox (-.5,-.2) rectangle (.5,.2);\draw (-.5,0) arc(-180:0:.5);\draw[->] (.05,-.5) -- +(.01,0);})$
and of 
$B:=\cA(\tikz[scale=.5]{\useasboundingbox (-.5,-.2) rectangle (.5,.2);\draw (-.5,0) arc(180:0:.5);\draw[->] (-.05,.5) -- +(-.01,0);})$
are each other's commutants (Haag duality).
The fusion of $H$ and $K$ is the completion of
\[
\Hom_B(H_0,H)\otimes_A K
\]
under the inner product 
$\langle\varphi\otimes \xi,\psi\otimes \eta\rangle:=\langle(\psi^*\varphi)\xi,\eta\rangle$, where $\psi^*\varphi\in\Hom_B(H_0,H_0)$ is identified with an element of $A$.
Equivalently, the fusion is the completion of
$H \otimes_B \Hom_A(H_0,K)$,
or the completion of 
$\Hom_B(H_0,H)\otimes_A H_0 \otimes_B \Hom_A(H_0,K)$.
\end{remark}

There is also a `coordinate free' version of the operation of fusion, that goes as follows.
Recall that a representation of a conformal net is a Hilbert space equipped with
compatible actions of the algebras $\cA(I)$ for all the subintervals of the standard circle.
More generally, for any circle $S$ (a \emph{circle} is an oriented $1$-manifold diffeomorphic to $S^1$) there is a notion of $S$-sector of $\cA$ 
that generalises that of a representation
\cite[Def 1.7]{BDH-cn1}:
\begin{definition}
Let $\cA$ be a conformal net.
An $S$-sector of $\cA$ is a Hilbert space $H$ and a collection of homomorphisms
\[
\rho_I \colon \cA(I) \to \B(H),\quad I\subsetneq S
\]
subject to the compatibility condition $\rho_I|_{\cA(J)}=\rho_J$ whenever $J \subset I$.
We write $\mathrm{Sect}_{S}(\cA)$ for the category of $S$-sectors of $\cA$.
\end{definition}

The category $\mathrm{Sect}_{S}(\cA)$ contains a distinguished object $H_0(S,\cA)$, well defined up to non-canonical isomorphism, called the \emph{vacuum sector}.\footnote{Since $H_0(S,\cA)$ is only well defined up to non-canonical isomorphism, it would be more correct to say \emph{a} vacuum sector as opposed to \emph{the} vacuum sector.}
By definition \cite[Def~1.17]{BDH-cn1},
for every interval $I\subset S$ and every orientation reversing involution $j:S\to S$ that fixes $\partial I$,
the vacuum sector $H_0(S,\cA)$ is isomorphic to $L^2(\cA(I))$ via an isomorphism
\begin{equation}\label{eq:   v:H_0(S,A) --> L^2(A(I))}
v:H_0(S,\cA)\to L^2(\cA(I)),
\end{equation}
well defined up to phase,
that intertwines the two left actions of $\cA(I)$ and satisfies
$v(\cA(j)(x)\xi)=(v(\xi))x$
for all $x\in \cA(I)$ and $\xi\in L^2(\cA(I))$.

Let $\Theta$ be any theta-graph, and let $S_1$, $S_2$, $S_3$ be its three circle subgraphs, oriented as follows:
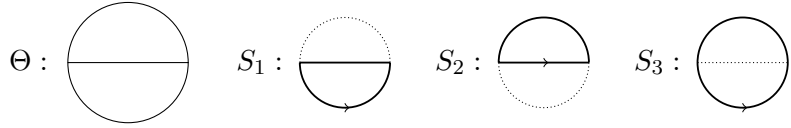
\begin{equation*} %\label{eq: Theta-graph}
\Theta:\tikzmath[scale=.8]{\useasboundingbox (-1.2,-1) rectangle (1.8,1);\draw circle (1) (-1,0) -- (1,0);}
S_1:\tikzmath[scale=.6]{\useasboundingbox (-1.2,-1) rectangle (2,1);\draw[line width=.7] (-1,0) arc (-180:0:1) (-1,0) -- (1,0);\draw[densely dotted](-1,0) arc (180:0:1);
\draw[->] (-90:1) ++ (0,0) -- +(.1,0);}
S_2:\tikzmath[scale=.6]{\useasboundingbox (-1.2,-1) rectangle (2,1);\draw[line width=.7] (-1,0) arc (180:0:1) (-1,0) -- (1,0);\draw[densely dotted](-1,0) arc (-180:0:1);
\draw[->] (0,0) -- +(.1,0);}
S_3:\tikzmath[scale=.6]{\useasboundingbox (-1.2,-1) rectangle (1.2,1);\draw[line width=.7] circle (1);\draw[densely dotted] (-1,0) -- (1,0);
\draw[->] (-90:1) ++ (0,0) -- +(.1,0);}
\end{equation*}
The circles $S_1$, $S_2$, $S_3$ are equipped with smooth structures which are compatible in the following sense:
around each of the trivalent vertices of $\Theta$, it is possible to pick a neighbourhood $Y\subset \Theta$, ($Y\cong\tikzmath[scale=.3]{\useasboundingbox (-1,-.8) -- (1,.8);\draw (0,0) -- (-30:1)(0,0) -- (90:.95)(0,0) -- (210:1);}$)
and local coordinates $\big\{f_i:[0,\varepsilon[\,\,\to Y
\big\}_{i=1,2,3}$ of the three legs of $Y$
so that for each pair $i,j \in \{1,2,3\}$ of distinct indices, the map $(-f_i)\cup f_j:\,\,]\!-\!\varepsilon,\varepsilon[\,\,\to Y$
is smooth when viewed as a map to the relevant circle.
Let $I:=S_1\cap S_2$ be the central interval, equipped with the orientation inherited from $I_2$.
Then Connes fusion along $\cA(I)$ defines a functor \cite[Def~1.31]{BDH-cn1}:
\begin{equation*} %\label{eq: theta fusion}
\boxtimes_{\cA(I)}:\mathrm{Sect}_{S_1}(\cA)\times \mathrm{Sect}_{S_2}(\cA)\to \mathrm{Sect}_{S_3}(\cA).
\end{equation*}
We denote the fusion of representations graphically by:
\begin{equation}\label{eq:3circs}
H\boxtimes K \,=\,
H\boxtimes_{\cA(\tikz[scale=.3]{\useasboundingbox (-.5,-.18) rectangle (.5,.2);
\draw (-.5,0) -- (.5,0);\draw[->] (.15,0) -- +(.01,0);
})}K\,=\,
\tikz[scale=.6]{
\useasboundingbox (-1.2,-.1) rectangle (1.2,.8);
\draw circle (1) (-1,0) -- (1,0);
\node at (0,-.45) {$H$};
\node at (0,.45) {$K$};
} 
\bigskip
\end{equation}

\subsection{The braiding}\label{sec: the braiding}

In the literature on algebraic quantum field theory, the braiding on $\Rep(\cA)$ is usually defined as follows \cite[\S2]{MR1016869}\cite[\S7]{MR1027496}\cite[Def~4.16]{MR1231644}.
First of all, the objects of $\Rep(\cA)$ are represented by \emph{localised endomorphisms} of the $*$-algebra
%\footnote{It is customary to take the colimit in C*-algebras; we prefer to take the colimit in $*$-algebras.}
\[
\mathfrak A:=\underset{I\subset \R}{\mathrm{colim}}\,\cA(I).
\]
Here, a $*$-algebra endomorphism $\rho:\mathfrak A\to \mathfrak A$ is said to be \emph{localised} in a bounded region $O\subset \R$ if it acts as the identity on the subalgebras $\cA(I)\subset \mathfrak A$ for every $I$ disjoint from~$O$.
(We refer the reader to \cite[\S IV]{MR1231644} for an explanation of the bijective correspondence between non-zero representations of $\cA$ and localised endomorphisms of $\mathfrak A$.) % of a chiral conformal net.)

Given localised endomorphisms $\rho_1$ and $\rho_2$, one picks unitaries $U_1,U_2\in\mathfrak A$ such that 
$\hat\rho_1:=\mathrm{Ad}(U_1)\rho_1$ is localised in a region which is to the right of
the region where $\hat\rho_2:=\mathrm{Ad}(U_2)\rho_2$ is localised.
The element
\begin{equation*} %\label{eq: def of varepsilon}
\varepsilon := \rho_2(U_1)^{-1}U_2^{-1}U_1\rho_1(U_2) %= U_2^{-1} \hat\rho_2(U_1)^{-1}\hat\rho_1(U_2)U_1
\end{equation*}
is then an intertwiner from $\rho_1\rho_2$ to $\rho_2\rho_1$ (it satisfies $\varepsilon \rho_1\rho_2(x)=\rho_2\rho_1(x) \varepsilon$),
which is called the \emph{braiding} of $\rho_1$ and $\rho_2$.
It is independent of the choice of unitaries $U_1$ and $U_2$,
provided $\mathrm{Ad}(U_1)\rho_1$ is localised to the right of the localisation region of $\mathrm{Ad}(U_2)\rho_2$.
Here, the intertwining property is best explained by noting that $\varepsilon$ is a composite of intertwiners
\[
\rho_1\rho_2
\xrightarrow{\rho_1(U_2)}
\rho_1\hat\rho_2
\xrightarrow{U_1}
\hat\rho_1\hat\rho_2=\hat\rho_2\hat\rho_1
\xrightarrow{U_2^{-1}}
\rho_2\hat\rho_1
\xrightarrow{\rho_2(U_1)^{-1}}
\rho_2\rho_1.
\]
%\[
%\rho_1\rho_2
%\xrightarrow{U_1}
%\hat\rho_1\rho_2
%\xrightarrow{\hat\rho_1(U_2)}
%\hat\rho_1\hat\rho_2=\hat\rho_2\hat\rho_1
%\xrightarrow{\hat\rho_2(U_1)^{-1}}
%\hat\rho_2\rho_1
%\xrightarrow{U_2^{-1}}
%\rho_2\rho_1.
%\]

We now adapt the above definition\footnote{Note that we also have $\varepsilon = U_2^{-1} \hat\rho_2(U_1)^{-1}\hat\rho_1(U_2)U_1$. It is that second formula which most closely resembles our working definition \eqref{eq: the braiding} of the braiding.} of the braiding to the case of the fusion product \eqref{eq:3circs}.
Given a representation $H$ and an element $x\in \cA(I)$ for some interval $I\subset S^1$, we write $\rho^H(x)$ for the action of $x$ on $H$.
Let $I_n:=\{e^{2\pi i \theta}:\theta\in[\tfrac{n-1}{4},\tfrac{n}{4}]\}$ for $n=1,2,3,4$:
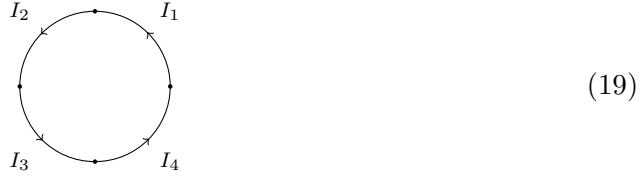
\begin{equation}\label{eq: pic: I_1, I_2, I_3, I_4}
\tikzmath{
\draw circle (1);
\fill (0:1) circle (.03)  (90:1) circle (.03)  (180:1) circle (.03)  (270:1) circle (.03);
\node at (1,1) {$\scriptstyle I_1$};
\node at (-1,1) {$\scriptstyle I_2$};
\node at (-1,-1) {$\scriptstyle I_3$};
\node at (1,-1) {$\scriptstyle I_4$};
\draw[->] (45:1) -- +(-.01,.012);
\pgftransformrotate{90}
\draw[->] (45:1) -- +(-.01,.012);
\pgftransformrotate{90}
\draw[->] (45:1) -- +(-.01,.012);
\pgftransformrotate{90}
\draw[->] (45:1) -- +(-.01,.012);}
\end{equation}
Let us adopt the notations $I_{123}:=I_1\cup I_2\cup I_3$, $I_{234}:=I_2\cup I_3\cup I_4$, $I_{34}=I_3\cup I_4$, etc.
By definition, the fusion of representations is given by
$H\boxtimes K=H\boxtimes_{\cA(I_{34})} K$.

Let $\varphi_1:I_{341}\to I_{341}$ be a diffeomorphism that sends $1$ to $-i$ and is the identity near the boundary of that interval.
Given two representations $H,K\in\Rep(\cA)$, we define
\begin{equation}\label{eq:   H box^phi K}
H\boxtimes^{\varphi_1} K:=H\,\boxtimes_{\cA(I_{34})} ({}^{\varphi_1\!}K),
\end{equation}
where ${}^{\varphi_1\!}K$ is the Hilbert space $K$ with action of $\cA(I_{34})$ twisted by $\cA(\varphi_1):\cA(I_{34})\to \cA(I_3)$.
We equip $H\boxtimes^{\varphi_1} K$ with the following actions of $\cA(I_{412})$ and of $\cA(I_3)$.
The algebra $\cA(I_{412})$ acts on $H\boxtimes^{\varphi_1} K$ by means of its usual action on $K$.
The algebra $\cA(I_3)$ acts on $H\boxtimes^{\varphi_1} K$ by first applying $\cA(\varphi_1)^{-1}:\cA(I_3)\to \cA(I_{34})$ and then using the action of $\cA(I_{34})$ on $H$.
We will see later that those actions extend, by strong additivity, to the structure of a representation on $H\boxtimes^{\varphi_1} K$.
%If on the other hand $I\subset I_{3}$, then we apply $\cA(\varphi_1)^{-1}:\cA(I)\to \cA(I_{34})$ and use the action of $\cA(I_{34})$ on $K$.
%For intervals which are contained in neither $I_{412}$ nor $I_3$, we use the strong additivity axiom to piece together the actions of $\cA(I\cap I_{412})$ and $\cA(I\cap I_3)$ into an acton of $\cA(I)$ on $H\boxtimes^{\varphi_1} K$.
%The fact that this is possible is a consequence of Lemma~\ref{lem:  U_1^(H,K) } below.

Pick a unitary $u_1\in\cA(I_{341})$ such that $\mathrm{Ad}(u_1)=\cA(\varphi_1)$ (Definition~\ref{Def:CN}.{\it iv}).

\begin{lemma}\label{lem:  U_1^(H,K) }
The isomorphism
\begin{equation}\label{eq: def: U_1}
U_1=U_1^{(H,K)}:=\big(\id_{H}\boxtimes \rho^K(u_1)\big)\circ\rho^{H\boxtimes K}(u_1)^{-1}:H\boxtimes K\to H\boxtimes^{\varphi_1} K
\end{equation}
intertwines the actions of $\cA(I_3)$ and of $\cA(I_{412})$.
%is a morphism of representations.
\end{lemma}

\begin{proof}
%To simplify the notation, we 
We write $\varphi$, $u$, and $U$ in place of $\varphi_1$, $u_1$, and $U_1$.
%We need to show that $U\circ \rho^{H\boxtimes K}(x)=\rho^{H\boxtimes^\varphi K}(x)\circ U$ holds for every $x\in \cA(I)$ and every $I\subset S^1$.
%By strong additivity, it is enough to check this for elements $x\in \cA(I_3)$ and $x\in \cA(I_{412})$.
For $x\in \cA(I_3)$, we have:
\begin{align*}
U\circ\rho^{H\boxtimes K}(x)
&= \big(\id_{H}\boxtimes \rho^K(u)\big)\circ\rho^{H\boxtimes K}(u^{-1}x)\\
&= \big(\id_{H}\boxtimes \rho^K(u)\big)\circ\rho^{H\boxtimes K}\big(\cA(\varphi)^{-1}(x)u^{-1}\big)\\
&= \big(\id_{H}\boxtimes \rho^K(u)\big)\circ\rho^{H\boxtimes K}\big(\cA(\varphi)^{-1}(x)\big)\circ\rho^{H\boxtimes K}(u^{-1})\\
&= \big(\id_{H}\boxtimes \rho^K(u)\big)\circ\big(\rho^{H}(\cA(\varphi)^{-1}(x))\boxtimes\id_K\big)\circ\rho^{H\boxtimes K}(u^{-1})\\
&= \big(\rho^{H}(\cA(\varphi)^{-1}(x))\boxtimes\id_K\big)\circ\big(\id_{H}\boxtimes \rho^K(u)\big)\circ\rho^{H\boxtimes K}(u^{-1})\\
&=\rho^{H\boxtimes^\varphi K}(x)\circ U.
\end{align*}
For $x\in \cA(I_{412})$, we have:
\begin{equation}\label{eq: long computation 2}\vspace{-.8cm}
\begin{split}
U\circ\rho^{H\boxtimes K}(x)
&= \big(\id_{H}\boxtimes \rho^K(u)\big)\circ\rho^{H\boxtimes K}(u^{-1}x)\\
&= \big(\id_{H}\boxtimes \rho^K(u)\big)\circ\rho^{H\boxtimes K}\big(\cA(\varphi)^{-1}(x)u^{-1}\big)\\
&= \big(\id_{H}\boxtimes \rho^K(u)\big)\circ\rho^{H\boxtimes K}\big(\cA(\varphi)^{-1}(x)\big)\circ\rho^{H\boxtimes K}(u^{-1})\\
&= \big(\id_{H}\boxtimes \rho^K(u)\big)\circ\big(\id_H\boxtimes\rho^{K}(\cA(\varphi)^{-1}(x))\big)\circ\rho^{H\boxtimes K}(u^{-1})\\
&= \big(\id_{H}\boxtimes \rho^K(u\cA(\varphi)^{-1}(x))\big)\circ\rho^{H\boxtimes K}(u^{-1})\\
&= \big(\id_{H}\boxtimes \rho^K(xu)\big)\circ\rho^{H\boxtimes K}(u^{-1})\\
&=\rho^{H\boxtimes^\varphi K}(x)\circ U.%\qedhere
\end{split}
\end{equation}
\end{proof}

\begin{corollary}
The actions of $\cA(I_3)$ and $\cA(I_{412})$ on $H\boxtimes^{\varphi_1} K$ endow it with the structure of an object of $\Rep(\cA)$.
The map \eqref{eq: def: U_1} is an isomorphism of representations.
\end{corollary}

Let us now consider a diffeomorhpism $\varphi_2:I_{234}\to I_{234}$ that sends $-1$ to $-i$ and is the identity near the boundary,
and let $u_2\in\cA(I_{341})$ be such that $\mathrm{Ad}(u_2)=\cA(\varphi_2)$.
We can then define $H\boxtimes^{\varphi_2}K$ analogously to \eqref{eq:   H box^phi K},
and we have an isomorphism of representations
\begin{equation}\label{eq:   def: U_2}
U_2=U_2^{(H,K)}:=\big(\id_{H}\boxtimes \rho^K(u_2)\big)\circ\rho^{H\boxtimes K}(u_2)^{-1}:H\boxtimes K\to H\boxtimes^{\varphi_2} K.
\end{equation}

We are now ready to translate the classical definition of the braiding into the language of Connes fusion:
\begin{definition}\label{def: braiding}
Given two representations $H$ and $K$,
the braiding isomorphism
\[
\beta_{H,K}:H\boxtimes K\to K\boxtimes H
\]
is the composite
\begin{equation}\label{eq: the braiding}
\begin{split}
\beta_{H,K}\,:\,\,H \boxtimes K\,\cong\,\,\,& H\boxtimes K \boxtimes H_0\\
\xrightarrow{\,U_1 \,{\scriptscriptstyle \boxtimes}\, \id\,}
\,\,& H\boxtimes^{\varphi_1}  K \boxtimes H_0\\
\xrightarrow{\,\id \,{\scriptscriptstyle \boxtimes}\, U_2\,}
\,\,& H\boxtimes^{\varphi_1} (K \boxtimes^{\varphi_2} H_0) \,\cong\, K \boxtimes^{\varphi_2} (H\boxtimes^{\varphi_1} H_0)\\
\xrightarrow{\,\id \,{\scriptscriptstyle \boxtimes}\,U_1^*\,}
\,\,& K \boxtimes^{\varphi_2}  (H\boxtimes H_0)\\
\xrightarrow{\,U_2^* \,{\scriptscriptstyle \boxtimes}\, \id\,}
\,\,& K \boxtimes (H\boxtimes H_0) \,\cong\, K \boxtimes H,
\end{split}
\end{equation}
where $H_0$ denotes the vacuum sector of the conformal net. The middle isomorphism is explained below.
\end{definition}
Pictorially, we like to represent the isomorphisms \eqref{eq: the braiding} as the following sequence of moves:
\begin{equation*} %\label{pic: the braiding}
\beta_{H,K}\,:\,\,\tikzmath[scale=.55]{\draw (0,0) circle (1) (-1,0) -- (1,0); \node at (0,-.5) {$\scriptstyle H$};\node at (0,.5) {$\scriptstyle K$}; }
\,\,\cong\,\,
\tikzmath[scale=.55]{\draw (-1,0) arc (180:0: 1 and 1.4)(-1,0) to[in=180-60, out =60](1,0)(-1,0) to[in=180+60, out =-60](1,0)(-1,0) arc (180:0: 1 and -1.4);
\node at (0,-.9) {$\scriptstyle H$};\node at (0,0) {$\scriptstyle K$};\node at (0,.9) {$\scriptstyle H_0$};}
\,\,\to\,\,
\tikzmath[scale=.55, xscale=-1]{\useasboundingbox (1,1.4) rectangle (-1,-1.4);
\draw (-1,0) arc (180:0: 1 and 1.4)(-1,0) to[in=180-60, out =60](1,0)(-1,0) to[in=180+60, out =-60](1,0);
\node[scale=.9] at (.47,-.78) {$\scriptstyle H$};\node at (0,0) {$\scriptstyle K$};\node at (0,.9) {$\scriptstyle H_0$};
\draw (0,-.5) to[in=-90, out =-82, looseness=2.7](1,0);}
\,\,\to\,\,
\tikzmath[scale=.55, xscale=-1]{\useasboundingbox (1,1.4) rectangle (-1,-1.4);
\draw (-1,0) arc (180:0:1) -- (-1,0); \draw (0,0) to[in=-90, out =-87, looseness=2.8](1,0); \draw (0,0) to[in=-90, out =-93, looseness=2.8](-1,0);
\node at (0,.5) {$\scriptstyle H_0$}; \node[scale=.9] at (.52,-.4) {$\scriptstyle H$}; \node[scale=.9] at (-.52,-.4) {$\scriptstyle K$};}
\,\,\to\,\,
\tikzmath[scale=.55, xscale=-1]{\useasboundingbox (1,1.4) rectangle (-1,-1.4);
\draw (-1,0) arc (180:0: 1 and 1.4)(-1,0) to[in=180-60, out =60](1,0)(-1,0) to[in=180+60, out =-60](1,0);
\node[scale=.9] at (-.45,-.78) {$\scriptstyle K$};\node at (0,0) {$\scriptstyle H$};\node at (0,.9) {$\scriptstyle H_0$};
\draw (0,-.5) to[in=-90, out =-98, looseness=2.7](-1,0);}
\,\,\to\,\,
\tikzmath[scale=.55]{\draw (-1,0) arc (180:0: 1 and 1.4)(-1,0) to[in=180-60, out =60](1,0)(-1,0) to[in=180+60, out =-60](1,0)(-1,0) arc (180:0: 1 and -1.4);
\node at (0,-.9) {$\scriptstyle K$};\node at (0,0) {$\scriptstyle H$};\node at (0,.9) {$\scriptstyle H_0$};}
\,\,\cong\,\,
\tikzmath[scale=.55]{\draw (0,0) circle (1) (-1,0) -- (1,0); \node at (0,-.5) {$\scriptstyle K$};\node at (0,.5) {$\scriptstyle H$};}
\end{equation*}

The isomorphism
$ %\begin{equation}\label{eq: middle iso}
H\boxtimes^{\varphi_1} (K \boxtimes^{\varphi_2} H_0) \cong K \boxtimes^{\varphi_2} (H\boxtimes^{\varphi_1} H_0)
$ %\end{equation}
which occurs in \eqref{eq: the braiding} 
requires some explanation.
Recall that for a right module $X$ and a left module $Y$
there is symmetry isomorphism $s:X\boxtimes_R Y\cong Y\boxtimes_{R^{\op}} X$ (Remark~\ref{rm: HK = KH}).
The isomorphism in the middle of \eqref{eq: the braiding} is the composite
\begin{equation}\label{eq: s-a-s}
\begin{split}
H\boxtimes_R {}^{\varphi_1} (K \boxtimes_R {}^{\varphi_2} &H_0) 
\xrightarrow s
{}^{\varphi_1} (K \boxtimes_R {}^{\varphi_2} H_0)\boxtimes_{R^\op} H
\\&\xrightarrow a
K \boxtimes_R {}^{\varphi_2}({}^{\varphi_1} H_0\boxtimes_{R^\op} H)
\xrightarrow {\id\boxtimes s}
K \boxtimes_R{}^{\varphi_2} (H\boxtimes_R{}^{\varphi_1} H_0),
\end{split}
\end{equation}
where the arrow labelled $a$ is the associator of Connes fusion.

\begin{remark}
At this point, it is not clear whether the braiding \eqref{eq: the braiding} depends on the choice of diffeomorphisms $\varphi_1$ and $\varphi_2$,
or whether it depends on our convention to add the vacuum sector on the top as opposed to the bottom.
In Section~\ref{sec: The Drinfel'd center}, we will show that it is independent of all these choices, by using the fact that it extends to the case when $H$ and $K$ are solitons ($H\in T^-_\cA$, $K\in T^+_\cA$, see Corollary~\ref{cor: One beta}).
\end{remark}

\begin{lemma}\label{lem: property of U}
Let $H$, $K$, and $L$ be representations.
Then the following diagram is commutative:
\begin{equation}\label{eq: property of U}
\tikzmath {
\node(1) at(0,1.4) {$(H\boxtimes K)\boxtimes L$};
\node(2) at(5.5,1.4) {$(H\boxtimes^{\varphi_1} K)\boxtimes L$};
\node(3) at(0,0) {$H\boxtimes (K\boxtimes L)$};
\node(4) at(5.5,0) {$H\boxtimes^{\varphi_1} (K\boxtimes L)$.\!};
\draw[->] (1) --node[above]{$\scriptstyle U_1^{(H,K)}\;\!\boxtimes\;\! \id_L$} (2);
\draw[->] (1) --node[left]{$\scriptstyle a$} (3);
\draw[->] (2) --node[right]{$\scriptstyle a$} (4);
\draw[->] (3) --node[above]{$\scriptstyle U_1^{(H,K\boxtimes L)}$} (4);  %
}
\end{equation}
A corresponding property holds for $U_2$.
\end{lemma}

\begin{proof}
The maps involved only depend on $L$ as an $\cA(I_{341})$-module.
Let $M:=\{L\in \cA(I_{341})$-Mod $:$ \eqref{eq: property of U} holds$\}$.
That category contains the vacuum sector $H_0$ as, in that case, the horizontal arrows in \eqref{eq: property of U} can be both identified with the map
$U_1^{(H,K)}:H\boxtimes K\to H\boxtimes^{\varphi_1} K$.
The category $M$ is closed under taking direct sums and taking direct summands.
$H_0$ generates $\cA(I_{341})$-Mod under those operations.
Therefore $M=\cA(I_{341})$-Mod.
\end{proof}

\begin{lemma}\label{lem: cocycle property of U}
Let $H$, $K$, and $L$ be representations.
Then the following diagram is commutative:
\[
\qquad\tikzmath {
\node(1) at(0,1.5) {$H\boxtimes K\boxtimes L$};
\node(2) at(5.5,1.5) {$H\boxtimes^{\varphi_1} (K\boxtimes L)$};
\node(3) at(0,0) {$(H\boxtimes K)\boxtimes^{\varphi_1} L$};
\node(4) at(5.5,0) {$H\boxtimes^{\varphi_1} (K\boxtimes^{\varphi_1} L)$.\!};
\draw[->] (1) --node[above]{$\scriptstyle %U_1^{(H,K)}\;\!\boxtimes\;\! \id_{L}\;=\;
U_1^{(H,K\boxtimes L)}$} (2);
\draw[->] (1) --node[left]{$\scriptstyle U_1^{(H\boxtimes K,L)}$} (3);
\draw[->] (2) --node[right]{$\scriptstyle \id_H\;\!\boxtimes\;\! U_1^{(K,L)}$} (4);
\draw[->] (3) --node[above]{$\scriptstyle a$} (4);  %
}
\]
A corresponding property holds for $U_2$.
\end{lemma}

\begin{proof}
By definition,
\begin{align}
\notag
&\big(\id_H\boxtimes U_1^{(K,L)}\big)\circ U_1^{(H, K\boxtimes L)}\\
\notag=\,&
\id_{H}\boxtimes \Big(\big(\id_{K}\boxtimes \rho^L(u_1)\big)\circ\rho^{K\boxtimes L}(u_1)^{-1}\Big)
\circ
\Big(\big(\id_{H}\boxtimes \rho^{K\boxtimes L}(u_1)\big)\circ\rho^{H\boxtimes K\boxtimes L}(u_1)^{-1}\Big)\\
%\notag=\,&
%\big(\id_{H\boxtimes K}\boxtimes \rho^L(u_1)\big)\circ\big(\id_{H}\boxtimes \rho^{K\boxtimes L}(u_1)\big)^{-1}
%\circ
%\big(\id_{H}\boxtimes \rho^{K\boxtimes L}(u_1)\big)\circ\rho^{H\boxtimes K\boxtimes L}(u_1)^{-1}\\
=\,&																						\label{eq: cocycle property for U_1}
\big(\id_{H\boxtimes K}\boxtimes \rho^{L}(u_1)\big)\circ\rho^{H\boxtimes K\boxtimes L}(u_1)^{-1}\\
\notag=\,&
U_1^{(H\boxtimes K, L)}.\qedhere
\end{align}
\end{proof}

\begin{proposition}\label{prop: two `hexagon' axioms}
The isomorphism \eqref{eq: the braiding} satisfies the `hexagon' axioms
\[
\beta_{H,K\boxtimes L}=(\id_K\boxtimes \beta_{H,L})(\beta_{H,K}\boxtimes \id_L)
\quad\text{and}\quad
\beta_{H\boxtimes K, L}=(\beta_{H,L}\boxtimes \id_K)(\id_H\boxtimes \beta_{K,L})
\]
(we omit the associators for brevity).
\end{proposition}

\begin{proof}
We only prove the first axiom.
To keep notations short, we drop the symbol $\boxtimes$, 
we write $H_1K$ for $H\boxtimes^{\varphi_1}K$, we write $H_2K$ for $H\boxtimes^{\varphi_2}K$,
and we omit the vacuum sector $H_0$.
The definition \eqref{eq: the braiding} of the braiding then becomes:
\[
\beta_{H,K}:HK\to H_1K\to H_1K_2\to K_2H_1\to K_2H\to KH,
\]
where the isomorphism $H_1K_2\to K_2H_1$ is the map constructed in \eqref{eq: s-a-s}.
Consider the following diagram, where all expressions are associated to the right unless otherwise indicated
(for example, $HKL$ stands for $H\boxtimes (K \boxtimes  (L \boxtimes H_0))$, $H_1K_2L_2$ stands for $H\boxtimes^{\varphi_1} (K \boxtimes^{\varphi_2}  (L \boxtimes^{\varphi_2} H_0))$, and $H_1(KL)_2$ stands for $H\boxtimes^{\varphi_1}\!((K\boxtimes L)\boxtimes^{\varphi_2}\!H_0)$):
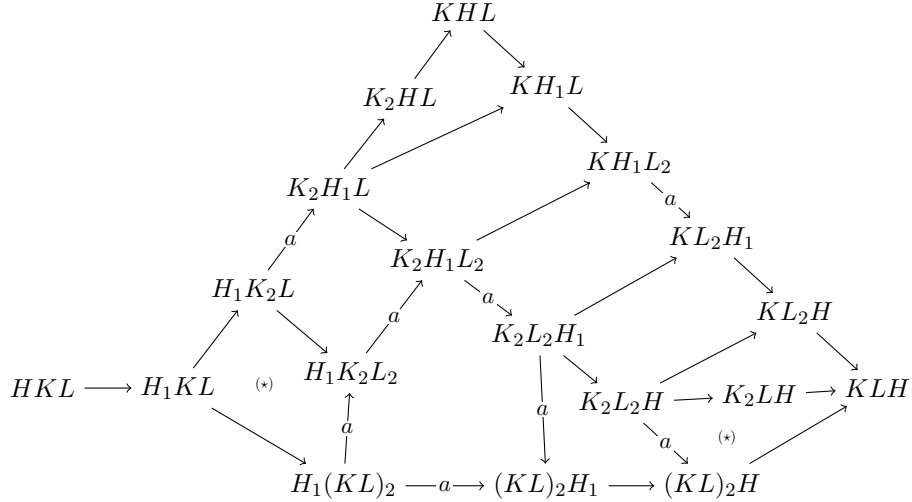
\[
\tikzmath[yscale=-1]{
\node[scale=.9](3) at (2-.2,1.7) {$H_1 K_2 L $};
\node[scale=.9](4) at (3-.2,0.35) {$K_2 H_1 L $};
\node[scale=.9](5) at (4-.25,-.85) {$K_2 H L $};
\node[scale=.9](6) at (5-.4,-2) {$K H L $};
\node[scale=.9](7) at (6-.3,-1) {$K H_1 L $};
\node[scale=.9](8) at (7-.2,0) {$K H_1 L_2 $};
\node[scale=.9](9) at (8-.1,1) {$K L _2 H_1 $};
\node[scale=.9](10) at (9,2) {$K L _2 H $};

\node[scale=.9](A) at (-1,3) {$H K L $};
\node[scale=.9](B) at (.8,3) {$H_1 K L $};
\node[scale=.9](C) at (3,4.3) {$H_1 (K L)_2 $};
\node[scale=.9](D) at (5.7,4.3) {$(K L)_2 H_1 $};
\node[scale=.9](E) at (7.9,4.3) {$(K L)_2 H $};
\node[scale=.9](F) at (10.1,3) {$K L H $};

\node[scale=.9](a) at (3.1,2.8) {$H_1 K_2 L_2 $};
\node[scale=.9](b) at (4.25,1.3) {$K_2 H_1 L_2 $};
\node[scale=.9](c) at (5.6,2.3) {$K_2 L_2 H_1 $};
\node[scale=.9](d) at (6.7,3.2) {$K_2 L_2 H $};
\node[scale=.9](e) at (8.53,3.1) {$K_2 L H $};

\node[scale=.5] at (1.95,2.95) {$(\star)$};
\node[scale=.5] at (8.1,3.65) {$(\star)$};

\draw[->] (d) -- (e);
\draw[->] (e) -- (F);
\draw[->] (3) --node[fill=white, inner sep=1, scale=.8]{$a$} (4);  %
\draw[->] (4) -- (5);\draw[->] (5) -- (6);\draw[->] (6) -- (7);\draw[->] (7) -- (8);
\draw[->] (8) --node[fill=white, inner sep=1, scale=.8]{$a$} (9);  %
\draw[->] (9) -- (10);\draw[->] (10) -- (F);

\draw[->] (4) -- (7);
\draw[->] (A) -- (B);\draw[->] (B) -- (C);
\draw[->] (C) --node[fill=white, inner sep=1, scale=.8]{$a$} (D);   %
\draw[->] (D) -- (E);\draw[->] (E) -- (F);
\draw[->] (a) --node[fill=white, inner sep=1, scale=.8]{$a$} (b);  %
\draw[->] (3) -- (a);
\draw[->] (B) -- (3);
\draw[->] (C) --node[fill=white, inner sep=1, scale=.8]{$a$} (a);   %
\draw[->] (c) --node[fill=white, inner sep=1, scale=.8]{$a$} (D);   %
\draw[->] (4) -- (b);
\draw[->] (d) -- (10);
\draw[->] (c) -- (d);
\draw[->] (d) --node[fill=white, inner sep=1, scale=.8]{$a$} (E);  %
\draw[->] (b) --node[fill=white, inner sep=1, scale=.8]{$a$} (c);  %
\draw[->] (b) -- (8);
\draw[->] (c) -- (9);}
\]
The arrows labelled $a$ involve associators.

One reads $\beta_{H,K}\boxtimes \id_L$ along the top left (this is a consequences of Lemma~\ref{lem: property of U} given our convention that all expressions are associated to the right), one reads $\id_K\boxtimes \beta_{H,L}$ along the top right, and one reads $\beta_{H,K\boxtimes L}$ along the bottom.
In order to show that the desired equation $\beta_{H,K\boxtimes L}=(\id_K\boxtimes \beta_{H,L})(\beta_{H,K}\boxtimes \id_L)$ holds, it is therefore enough to argue that each individual cell in the above diagram is commutative.
The commutativity of the cells marked by a little star is the content of Lemma~\ref{lem: cocycle property of U}.
The other cells are easily seen to be commutative.
\end{proof}
%the cocycle property
%\[
%\qquad\tikzmath {
%\node(1) at(0,1.5) {$K\boxtimes L\boxtimes H$};
%\node(2) at(7,1.5) {$K\boxtimes^{\varphi_2} L\boxtimes H$};
%\node(3) at(0,0) {$(K\boxtimes L)\boxtimes^{\varphi_2} H$};
%\node(4) at(7,0) {$K\boxtimes^{\varphi_2} L\boxtimes^{\varphi_2} H$};
%\draw[->] (1) --node[above]{$\scriptstyle U_2^{(K,L)}\;\!\boxtimes\;\! \id_{H}\;=\;U_2^{(K,L\boxtimes H)}$} (2);
%\draw[->] (1) --node[left]{$\scriptstyle U_2^{(K\boxtimes L,H)}$} (3);
%\draw[->] (2) --node[right]{$\scriptstyle \id_K\boxtimes^{\varphi_2}U_2^{(K\boxtimes L,H)}$} (4);
%\draw[->] (3) --node[fill=white, inner sep=1, scale=.8]{$a$} (4);  %
%}
%\]
%which follows readily from the definition \eqref{eq:   def: U_2} of $U_2$.

\section{Solitons}

Let $\cA$ be a conformal net.
Recall that a \emph{soliton} is a Hilbert space (always assumed separable) equipped with compatible actions of the algebras $\cA(I)$, for all subintervals of $S^1$ whose interior does not contain the point $1$.
We write $T_\cA=T^+_\cA$ for the category of solitons.
The category $T^-_\cA$ is defined similarly, using the point $-1$ instead of the point $1$.
We also call the elements of $T^-_\cA$ solitons, when this creates no confusion.

Our first goal is to identify the categories $T^+_\cA$ and $T^-_\cA$ with subcategories of $\Bim(R)$.

\subsection{Solitons as bimodules}\label{sec: solitons as bimodules}

Let $I_1,\ldots,I_4$ be the subintervals of $S^1$ depicted in \eqref{eq: pic: I_1, I_2, I_3, I_4}, and let us adopt the same notations as in the previous section: $I_{12}=I_1\cup I_2$, $I_{23}=I_2\cup I_3$, etc.
Let $I_0$ be the standard interval which we use to parametrize the lower and upper halves of the standard circle, as in \eqref{eq: identifications}.
Let $\cA$ be a conformal net, and let $R:=\cA(I_0)$.

The parametrizations $I_0\to I_{34}$ and $\bar I_0\to I_{12}$ %described in \eqref{eq: identifications} 
induce an equivalence of categories between the category $\Bim(R)$ and the category whose objects are separable Hilbert spaces equipped with commuting left actions
of the algebras $\cA(I_{12})$ and $\cA(I_{34})$.
This allows us to identify $T^+_\cA$, $T^-_\cA$, and $\Rep(\cA)$ with subcategories of $\Bim(R)$:
\begin{align} \label{eq: three iotas}
\notag 
&\,\,\iota^+:\,T^+_\cA\;\to\; \Bim(R)\\
&\,\,\iota^-:\,T^-_\cA\;\to\; \Bim(R)\\ %\text{and}\quad
\notag 
&\iota:\,\Rep(\cA)\to \Bim(R).
\end{align}

\begin{lemma}\label{lem: T --> Bim(R) is fully faithful}
The functors $\iota^+$, $\iota^-$, and $\iota$ are fully faithful.
\end{lemma}

\begin{proof}
The functors $\iota^+$, $\iota^-$, $\iota$ are clearly faithful. We prove that they are full.
Let $H,K\in T^+(\cA)$ (respectively $H,K\in T^-(\cA)$, respectively $H,K\in\Rep(\cA)$), and let $f:H\to K$ be a morphism in $\Bim(R)$.
By definition, $f$ commutes with the actions of $\cA(I_{12})$ and of $\cA(I_{34})$.
Let $I\subset S^1$ be an interval which does not contain $1$ in its interior (respectively an interval such that $-1\not\in\mathring I$, respectively any subinterval of $S^1$).
By assumption, $f$ commutes with $\cA(I\cap I_{12})$ and $\cA(I\cap I_{34})$.
By strong additivity (Definition~\ref{Def:CN}.{\it ii}\;\!), these two algebras generate a dense subalgebra of $\cA(I)$.
So $f$ commutes with $\cA(I)$.
This being true for any $I$, $f$ is a morphism in $T^+(\cA)$ (respectively a morphism in $T^-(\cA)$, respectively a morphism in $\Rep(\cA)$).
\end{proof}

For a Hilbert space equipped with commuting left actions of $\cA(I_{12})$ and of $\cA(I_{34})$, consider the following properties:\smallskip
\begin{enumerate}
\item[(a)] The actions of $A(I_2)$ and $A(I_3)$ extend to an action of $A(I_{23})$.
\item[(b)] The actions of $A(I_4)$ and $A(I_1)$ extend to an action of $A(I_{41})$.\smallskip
\end{enumerate}
By using the parametrizations~\eqref{eq: identifications},
%transferring along the equivalence of categories described above, 
we may treat (a) and (b) % above two conditions 
as conditions on $R$-$R$-bimodules.

\begin{lemma}\label{lem: essential images}
The essential images of the functors $\iota^+$, $\iota^-$, and $\iota$ %\eqref{eq: three iotas}
%$\iota^+:T^+_\cA\to \Bim(R)$, $\iota^-:T^-_\cA\to \Bim(R)$, and $\iota:\Rep(\cA)\to \Bim(R)$
are given~by:
\begin{align*}
&\mathrm{Im}(\iota^+)=\{H\in\Bim(R)\,|\, \text{\rm condition (a) holds}\},\\
&\mathrm{Im}(\iota^-)=\{H\in\Bim(R)\,|\, \text{\rm condition (b) holds}\},\\
&\mathrm{Im}(\iota)\,=\,\{H\in\Bim(R)\;|\; \text{\rm both (a) and (b) hold}\}.
\end{align*}
\end{lemma}

\noindent
In order to establish this lemma,
we will need the following technical result: % from \cite{Conformal-nets-are-factorization-algebras} and~\cite{BDH-cn1}:

\begin{lemma}\label{Lem.1.9 of CN1}
Let $M$ be a connected 1-manifold (either a circle or an interval), and
let $\{I_i\subset M\}_{i\in\cI}$ be a collection of intervals that satisfy
\[
\bigcup_{i\in\cI}I_i=M\qquad \text{and}\qquad \bigcup_{i\in\cI}\mathring I_i=\mathring M.
\]
Let $H$ be a Hilbert space equipped with actions $\rho_i:\cA(I_i)\to B(H)$ for $i\in \cI$
which are compatible in the sense that:
\begin{enumerate}
\item 
$\rho_i|_{\cA(I_i\cap I_j)}=\rho_j|_{\cA(I_i\cap I_j)}:\cA(I_i\cap I_j)\to B(H)$.
\item
For every $j,k\in{\cI}$ and every intervals $J\subset I_j$, $K\subset I_k$ with disjoint interiors, 
the algebras $\rho_j(\cA(J))$ and $\rho_k(\cA(K))$ commute.
\end{enumerate}
Then for every interval $I\subsetneq M$, the actions
\[
\rho_i|_{\cA(I\cap I_i)}:\cA(I\cap I_i)\to B(H)
\]
extend uniquely to an action of $\cA(I)$ on $H$. %\to B(H)}$.
\end{lemma}

\begin{proof}
The case when $M$ is a circle was proved in \cite[Lem.\,1.9]{BDH-cn1}.
The case when $M$ is an interval was proved in \cite[Lem.\,4]{Conformal-nets-are-factorization-algebras}. (And the two proofs are essentially the same.)
\end{proof}

\begin{proof}[Proof of Lemma~\ref{lem: essential images}]
Clearly, every $H\in T^+(\cA)$ satisfies (a), every $H\in T^-(\cA)$ satisfies (b), and every $H\in \Rep(\cA)$ satisfies both.

Let $S^1_{\mathrm{cut}}$ be the manifold described in Section~\ref{sec: Representations and solitons}.
If a bimodule $H\in\Bim(R)$ satisfies condition (a), then we may apply Lemma~\ref{Lem.1.9 of CN1} with $M=S^1_{\mathrm{cut}}$.
The Hilbert space $H$ admits actions of the algebras $\cA(I)$ for all $I\subsetneq S^1_{\mathrm{cut}}$, and is therefore a soliton.
The argument for $T^-(\cA)$ is identical.

If a bimodule $H\in \Bim(R)$ satisfies both (a) and (b), then we can apply Lemma~\ref{Lem.1.9 of CN1} with $M=S^1$.
The Hilbert space $H$ admits actions of the algebras $\cA(I)$ for all $I\subsetneq S^1$, and is therefore a representation of $\cA$.
\end{proof}

\begin{lemma}\label{lem: T_A and Rep(A) closed under fusion}
The subcategories $T^+_\cA$, $T^-_\cA$, and $\Rep(\cA)$ of $\Bim(R)$ are closed under Connes fusion.
\end{lemma}

\begin{proof}
In view of Lemma~\ref{lem: essential images}, it is enough to show that the properties (a) and (b) are preserved under fusion.
By symmetry, it is enough to treat just one of them.

Let $H$ and $K$ be bimodules that satisfy property (a).
Then, by \cite[Cor.~1.29]{BDH-cn1}, the actions of $\cA(I_2)$ on $K$ and $\cA(I_3)$ on $H$ extend to an action of $\cA(I_{23})$ on $H\boxtimes K$.
That is, $H\boxtimes K$ satisfies (a).
(The intervals which were denoted by $I$, $I_l$, $I_r$, and $I_l\circledast_II_r$ in \cite[Cor.~1.29]{BDH-cn1} correspond to $I_0$, $I_{123}$, $I_{234}$, and $I_{23}$, respectively.)
\end{proof}

\subsection{The braiding between $T^-_\cA$ and $T^+_\cA$}\label{sec: the braiding between T^- and T^+}

Given two solitons $H\in T^-_\cA$ and $K\in T^+_\cA$, we can use the inclusions \eqref{eq: three iotas} %$T^-_\cA\rightarrow \Bim(R)$ and $T^+_\cA\rightarrow \Bim(R)$
to define their fusion $H\boxtimes K\in\Bim(R)$.
The goal of this section is to extend the braiding on $\Rep(\cA)$ to a braiding
\[
\beta_{H,K}:H\boxtimes K\to K\boxtimes H
\]
which is defined for all $H\in T^-_\cA$ and $K\in T^+_\cA$.

Let $\varphi_1:I_{341}\to I_{341}$ and %$\varphi_2:I_{234}\to I_{234}$, 
%$u_1\in\cA(I_{341})$, and
\[
H\boxtimes^{\varphi_1} K:=H\boxtimes_R{}^{\varphi_1} K %\qquad \text{and}\qquad H\boxtimes^{\varphi_2} K:=H\boxtimes_R{}^{\varphi_2} K
\]
 %, and $u_2\in\cA(I_{341})$ 
be as in Section~\ref{sec: the braiding}.
Here, as before, 
${}^{\varphi_1\!}K$ denotes the Hilbert space $K$ with action of $\cA(I_{34})$ twisted by $\cA(\varphi_1):\cA(I_{34})\to \cA(I_{3})$.
We equip $H\boxtimes^{\varphi_1} K$ %and $H\boxtimes^{\varphi_2} K$ 
with the following actions of the algebras 
$\cA(I_{12})$, $\cA(I_3)$, and $\cA(I_4)$.
The algebras $\cA(I_{12})$ and $\cA(I_{4})$ act on $H\boxtimes^{\varphi_1} K$ by their usual action on $K$.
%The algebra $\cA(I_{4})$ acts on $H\boxtimes^{\varphi_1} K$ by its usual action on $K$.
%The algebra $\cA(I_{3})$ acts on $H\boxtimes^{\varphi_2} K$ by its usual action on $K$.
The algebra $\cA(I_3)$ acts by first applying $\cA(\varphi_1)^{-1}:\cA(I_3)\to \cA(I_{34})$ and then using the usual action of $\cA(I_{34})$ on $H$.
%The algebra $\cA(I_4)$ acts on $H\boxtimes^{\varphi_2} K$ by first applying $\cA(\varphi_2)^{-1}:\cA(I_4)\to \cA(I_{34})$ and then using the action of $\cA(I_{34})$ on $H$.
We find it useful to represent the Hilbert spaces $H\boxtimes K$ and $H\boxtimes^{\varphi_1} K$
by the following pictures: % and $H\boxtimes^{\varphi_2} K$:
\[
H\boxtimes K\,=\,\,
\tikzmath[scale=.65]{\draw (0,0) circle (1) (-1,0) -- (1,0); \node at (0,-.5) {$H$};\node at (0,.5) {$K$};
\node[fill=white, circle, inner sep=0] at (-1,-.135) {$\scriptscriptstyle \ast$};
\node[fill=white, circle, inner sep=0] at (1,.135) {$\scriptscriptstyle \ast$};
}
\qquad\qquad
H\boxtimes^{\varphi_1} K\,=\,\,
\tikzmath[scale=.65, xscale=-1]{\useasboundingbox (1,1.4) rectangle (-1,-1.4);
\draw (-1,0) arc (180:0:1) -- (-1,0); \draw (0,0) to[in=-90, out =-87, looseness=2.8](1,0);
\node at (0,.5) {$K$}; \node[scale=1.1] at (.52,-.4) {$\scriptstyle H$};
\node[fill=white, circle, inner sep=0] at (-1,.06) {$\scriptscriptstyle \ast$};
\node[fill=white, circle, inner sep=0] at (1,-.135) {$\scriptscriptstyle \ast$};
}
%\quad\qquad
%H\boxtimes^{\varphi_2} K\,=\,\,
%\tikzmath[scale=.65, xscale=-1]{\useasboundingbox (1,1.4) rectangle (-1,-1.4);
%\draw (-1,0) arc (180:0:1) -- (-1,0); \draw (0,0) to[in=-90, out =-93, looseness=2.8](-1,0);
%\node at (0,.5) {$K$}; \node[scale=1.1] at (-.52,-.4) {$\scriptstyle H$};
%\node[fill=white, circle, inner sep=0] at (-1,.135) {$\scriptscriptstyle \ast$};
%\node[fill=white, circle, inner sep=0] at (1,-.135) {$\scriptscriptstyle \ast$};
%}
\]
Here, the little star is a reminder that $H$ and $K$ are solitons, as opposed to representations.
We will see later, in Corollary~\ref{Corollary 37}, that the actions of $\cA(I_3)$ and $\cA(I_4)$ on $H\boxtimes^{\varphi_1} K$ extend, by strong additivity, to an action of $\cA(I_{34})$.

Let $u_1\in\cA(I_{341})$ be such that $\mathrm{Ad}(u_1)=\varphi_1$.
The unitary
\begin{equation} \label{eq: U_1 and U_2}
%\begin{split}
U_1^{(H,K)}=\big(\id_{H}\boxtimes \rho^K(u_1)\big)\circ\rho^{H\boxtimes K}(u_1)^{-1}:H\boxtimes K\to H\boxtimes^{\varphi_1} K
%\\
%U_2^{(H,K)}&:=\big(\id_{H}\boxtimes \rho^K(u_2)\big)\circ\rho^{H\boxtimes K}(u_2)^{-1}:H\boxtimes K\to H\boxtimes^{\varphi_2} K.
%\end{split}
\end{equation}
that was used in the definition \eqref{eq: the braiding} of the braiding
no longer makes sense when $H$ and $K$ are solitons, because the 
actions of $\cA(I_{34})$ and $\cA(I_1)$ on $H\boxtimes K$ might not extend to an action of $\cA(I_{341})$.
We circumvent this difficulty by a trick that is based on the following lemma:

\begin{lemma}[{\cite[Lem.\,B.24]{BDH-cn3}}]\label{lem: NT between module categories}
Let $R$ be a factor, let $A$ be any von Neumann algebra, and
let
\[
F,G:R\text{\rm -Mod} \to A\text{\rm -Mod}
\]
be completely additive functors.
Let $M\subset R\text{\rm -Mod}$ be a full subcategory with only one object, which is not the zero object.
%that has exactly one object.
%Let also assume that the unique object of $M$ is non-zero.

Then a natural transformation $\tau:F\to G$ is entirely determined by its restriction to $M$.
Conversely, any natural transformation $F|_M\to G|_M$ extends to a natural transformation $F\to G$. %, it is enough to specify it on the subcategory $M$.
\end{lemma}

\begin{proof}
By complete additivity, $\tau|_M$ determines $\tau$ on the subcategory of $R$-modules which are direct sums of the object of $M$.
Every $R$-module is a direct summand of one of the above form, so $\tau$ is determined on all of $R\text{\rm -Mod}$.
(The proof is even simpler when $R$ is a type $\mathrm{III}$ factor as, in that case, $M$ is equivalent to the subcategory of all non-zero modules.)
\end{proof}

Fix $H\in T^-_\cA$ and consider the functors $H\boxtimes -$ and $H\boxtimes^{\varphi_1}-$
from $\Rep(\cA)$ to the category $C$ whose objects are Hilbert spaces equipped with commuting actions of the algebras $\cA(I_3)$ and $\cA(I_4)$.
The definition \eqref{eq: U_1 and U_2} of $U_1^{(H,K)}:H\boxtimes K\to H\boxtimes^{\varphi_1} K$ makes sense in that context, so we get a natural transformation
\begin{equation*} %\label{eq: U_1 nat transf}
U_1^{(H,-)}:\Rep(\cA) \,\tworarrow\, C
\end{equation*}
from
\[
K\mapsto H\boxtimes K\qquad\text{to}\qquad K\mapsto H\boxtimes^{\varphi_1} K.
\]
We now observe that $H\boxtimes -$ and $H\boxtimes^{\varphi_1}-$ also make sense as functors from %the category 
$\cA(I_{34})\text{-Mod}$ to $C$.
Let $M\subset \cA(I_{34})\text{-Mod}$ be the full subcategory consisting of only the vacuum Hilbert space.

\begin{lemma}
The map $U_1^{(H,H_0)}:H\boxtimes H_0\to H\boxtimes^{\varphi_1} H_0$
defined in \eqref{eq: U_1 and U_2}
%(which is the value of \eqref{eq: U_1 nat transf} on $H_0\in\Rep(\cA)$)
is a natural transformation $M \tworarrow C$.
\end{lemma}

\begin{proof}
By Haag duality, $\mathrm{End}_M(H_0)=\mathrm{End}_{\cA(I_{34})\text{-Mod}}(H_0)=\cA(I_{12})$.
For every endomorphism $x\in \cA(I_{12})$ of $H_0$, we need to show that the diagram 
\[
\tikz{
\node(a) at (0,0) {$H\boxtimes H_0$};
\node(b) at (4,0) {$H\boxtimes^{\varphi_1} H_0$};
\node(c) at (0,-1.4) {$H\boxtimes H_0$};
\node(d) at (4,-1.4) {$H\boxtimes^{\varphi_1} H_0$};
\draw[->] (a) --node[above]{$\scriptstyle U_1^{(H,H_0)}$} (b);
\draw[->] (a) --node[left]{$\scriptstyle \id\,\boxtimes\, x$} (c);
\draw[->] (b) --node[right]{$\scriptstyle \id\,\boxtimes^{\varphi_1}\, x$} (d);
\draw[->] (c) --node[above]{$\scriptstyle U_1^{(H,H_0)}$} (d);
}
\]
commutes.
That computation was performed in \eqref{eq: long computation 2}.
\end{proof}

%$U\circ \rho^{H\boxtimes K}(x)=\rho^{H\boxtimes^\varphi K}(x)\circ U$ holds for every $x\in \cA(I)$ and every $I\subset S^1$.
%By strong additivity, it is enough to check this for elements $x\in \cA(I_3)$ and $x\in \cA(I_{412})$.

The category $C$ is of the form $A$-Mod for some von Neumann algebra (\cite[\S8]{MR0201989}).
By Lemma~\ref{lem: NT between module categories}, we therefore get:

\begin{corollary}\label{cor: NT whose value on the vacuum sector}
There exists a unique natural transformation
\begin{equation}\label{eq: NT defined on A(I_34)-Mod}
U_1^{(H,-)}:\cA(I_{34})\text{\rm -Mod} \,\tworarrow\; C
\end{equation}
whose value on the vacuum sector $H_0\in\cA(I_{34})\text{\rm -Mod}$ is given by the map \eqref{eq: U_1 and U_2}.
\end{corollary}

We now have two definitions of $U_1^{(H,-)}$ that we need to reconcile:

\begin{lemma}\label{lem: reconciliation}
Let $H\in T^-_\cA$ and $K\in \cA(I_{341})${\rm -Mod}. Then the map $U_1^{(H,K)}:H\boxtimes K\to H\boxtimes^{\varphi_1}K$ 
defined by \eqref{eq: U_1 and U_2}
agrees with the one given by Corollary~\ref{cor: NT whose value on the vacuum sector}.
\end{lemma}

\begin{proof}
Let $M\subset \cA(I_{341})$-Mod be the subcategory on which the two definitions of $U_1^{(H,K)}$ agree.
By definition, $M$ contains the vacuum sector.
Since $M$ is closed under direct sums and direct summands and
the vacuum sector generates $\cA(I_{341})$-Mod under those operations, $M=\cA(I_{341})$-Mod.
\end{proof}

We now restrict the natural transformation \eqref{eq: NT defined on A(I_34)-Mod} along the functor $T^+_\cA\to \cA(I_{34})\text{-Mod}$, to get a natural transformation
%\begin{equation*} %\label{eq: NT defined on T^+_A}
$U_1^{(H,-)}:T^+_\cA \,\tworarrow\, C$ % \mathsf{Hilb}
%\end{equation*}
from
$H\boxtimes -$ to $H\boxtimes^{\varphi_1} -$.

\begin{lemma}
Let $H\in T^-_\cA$ and $K\in T^+_\cA$ be solitons.
Then the map
\begin{equation}\label{fvjnsljnbsjbn}
U_1^{(H,K)}:H\boxtimes K\to H\boxtimes^{\varphi_1} K
\end{equation}
defined above intertwines the actions of $\cA(I_{12})$, $\cA(I_3)$, and $\cA(I_4)$.
\end{lemma}

\begin{proof}
The map \eqref{fvjnsljnbsjbn} intertwines the actions of $\cA(I_3)$ and $\cA(I_4)$ because it is a morphism in $C$.
Recall from \eqref{eq: NT defined on A(I_34)-Mod}
that $U_1^{(H,-)}$ is natural with respect to all morphisms of $\cA(I_{34})$-modules.
So the map $U_1^{(H,K)}:H\boxtimes K\to H\boxtimes^{\varphi_1} K$ intertwines the two actions of $\mathrm{End}_{\cA(I_{34})\text{-Mod}}(K)$.
The actions of $\cA(I_{12})$ on the source and on the target of \eqref{fvjnsljnbsjbn} factor through the aforementioned actions of $\mathrm{End}_{\cA(I_{34})\text{-Mod}}(K)$.
The map \eqref{fvjnsljnbsjbn} therefore intertwines the actions of $\cA(I_{12})$.
\end{proof}

\begin{corollary}\label{Corollary 37}
The actions of $\cA(I_3)$ and of $\cA(I_4)$ on $H\boxtimes^{\varphi_1} K$ extend, by strong additivity, to an action of $\cA(I_{34})$.
\end{corollary}

%We recapitulate what we have achieved so far.
Given two solitons $H\in T^-_\cA$ and $K\in T^+_\cA$,
we have upgraded $H\boxtimes^{\varphi_1} K$ to an object of $\Bim(R)$, and we have
made sense of the unitary $U_1^{(H,K)}:H\boxtimes K\to H\boxtimes^{\varphi_1} K$.
Similarly, given a diffeomorphism $\varphi_2:I_{234}\to I_{234}$ as in Section~\ref{sec: the braiding},
we can define $K\boxtimes^{\varphi_2} H$ and make sense of
$U_2^{(K,H)}:K\boxtimes H\to K\boxtimes^{\varphi_2} H$.
%We are finally in a position to adapt the definition of the braiding (Definition~\ref{def: braiding}) to the case of solitons:

\begin{definition}\label{def: soliton braiding}
Let $H\in T^-_\cA$ and $K\in T^+_\cA$ be solitons.
The braiding isomorphism $\beta_{H,K}:H\boxtimes K\to K\boxtimes H$
is the composite
\begin{equation}\label{eq: the braiding '}
\begin{split}
\beta_{H,K}\,:\,\,H \boxtimes K\,\cong\,\,\,& H\boxtimes K \boxtimes H_0\\
\xrightarrow{\,\,U_1 \,{\scriptscriptstyle \boxtimes}\, \id\,\,}
\,\,& H\boxtimes^{\varphi_1}  K \boxtimes H_0\\
\xrightarrow{\,\,\id \,{\scriptscriptstyle \boxtimes}\, U_2\,\,}
\,\,& H\boxtimes^{\varphi_1} (K \boxtimes^{\varphi_2} H_0) \,\cong\, K \boxtimes^{\varphi_2} (H\boxtimes^{\varphi_1} H_0)\\
\xrightarrow{\,\id \,{\scriptscriptstyle \boxtimes}\,U_1^{-1}\!}
\,\,& K \boxtimes^{\varphi_2}  (H\boxtimes H_0)\\
\xrightarrow{U_2^{-1} \,{\scriptscriptstyle \boxtimes}\, \id}
\,\,& K \boxtimes (H\boxtimes H_0) \,\cong\, K \boxtimes H,
\end{split}
\end{equation}
where $H_0$ denotes the vacuum sector.
\end{definition}

We represent this isomorphism graphically as follows:
\begin{equation*} %\label{pic: the braiding}
\beta_{H,K}\,:\,\,\tikzmath[scale=.55]{\draw (0,0) circle (1) (-1,0) -- (1,0); \node at (0,-.5) {$\scriptstyle H$};\node at (0,.5) {$\scriptstyle K$};
\node[fill=white, circle, inner sep=-.11] at (-1,-.15) {$\scriptscriptstyle \ast$};
\node[fill=white, circle, inner sep=-.11] at (1,.15) {$\scriptscriptstyle \ast$};
}
\,\,\cong\,\,
\tikzmath[scale=.55]{
\draw (-1,0) arc (180:0: 1 and 1.4) (175: 1 and -1.4) arc (175:0: 1 and -1.4);
\node at (0,-.9) {$\scriptstyle H$};\node at (0,0) {$\scriptstyle K$};\node at (0,.9) {$\scriptstyle H_0$};
\node[fill=white, circle, inner sep=-.11] at (-1.01,-.21) {$\scriptscriptstyle \ast$};
\node[fill=white, circle, inner sep=-.11] at (1,0) {$\scriptscriptstyle \ast$};
\draw(-1,0) to[in=180-50, out =65](1,.14)(-1,0) to[in=180+50, out =-65](1,-.14);
}
\,\,\to\,\,
\tikzmath[scale=.55, xscale=-1]{\useasboundingbox (1,1.4) rectangle (-1,-1.4);
\draw (-1,0) arc (180:0: 1 and 1.4);
\node[scale=.9] at (.47,-.78) {$\scriptstyle H$};\node at (0,0) {$\scriptstyle K$};\node at (0,.9) {$\scriptstyle H_0$};
\pgftransformxscale{-1}
\node[fill=white, circle, inner sep=-.11] at (-1.01,-.21) {$\scriptscriptstyle \ast$};
\node[fill=white, circle, inner sep=-.11] at (1,0) {$\scriptscriptstyle \ast$};
\draw(-1,0) to[in=180-45, out =60](1,.14)(-1,0) to[in=180+45, out =-60](1,-.14);
\draw (0,-.53) to[in=-80, out =-98, looseness=2.5](-.975,-.35);}
\,\,\to\,\,
\tikzmath[scale=.55, xscale=-1]{\useasboundingbox (1,1.4) rectangle (-1,-1.4);
\draw (-1,0) arc (180:0:1) -- (-1,0); \draw (0,0) to[in=-90, out =-87, looseness=2.8](1,0); \draw (0,0) to[in=-90, out =-93, looseness=2.8](-1,0);
\node at (0,.5) {$\scriptstyle H_0$}; \node[scale=.9] at (.52,-.4) {$\scriptstyle H$}; \node[scale=.9] at (-.52,-.4) {$\scriptstyle K$};
\node[fill=white, circle, inner sep=-.11] at (-1,-.15) {$\scriptscriptstyle \ast$};
\node[fill=white, circle, inner sep=-.11] at (1,-.15) {$\scriptscriptstyle \ast$};
}
\,\,\to\,\,
\tikzmath[scale=.55, xscale=-1]{\useasboundingbox (1,1.4) rectangle (-1,-1.4);
\draw (-1,0) arc (180:0: 1 and 1.4);
\node[scale=.9] at (-.45,-.78) {$\scriptstyle K$};\node at (0,0) {$\scriptstyle H$};\node at (0,.9) {$\scriptstyle H_0$};
\node[fill=white, circle, inner sep=-.11] at (-1.01,-.21) {$\scriptscriptstyle \ast$};
\node[fill=white, circle, inner sep=-.11] at (1,0) {$\scriptscriptstyle \ast$};
\draw(-1,0) to[in=180-45, out =60](1,.14)(-1,0) to[in=180+45, out =-60](1,-.14);
\draw (0,-.53) to[in=-80, out =-98, looseness=2.5](-.975,-.35);}
\,\,\to\,\,
\tikzmath[scale=.55, xscale=-1]{
\draw (-1,0) arc (180:0: 1 and 1.4) (175: 1 and -1.4) arc (175:0: 1 and -1.4);
\node at (0,-.9) {$\scriptstyle K$};\node at (0,0) {$\scriptstyle H$};\node at (0,.9) {$\scriptstyle H_0$};
\node[fill=white, circle, inner sep=-.11] at (-1.01,-.21) {$\scriptscriptstyle \ast$};
\node[fill=white, circle, inner sep=-.11] at (1,0) {$\scriptscriptstyle \ast$};
\draw(-1,0) to[in=180-50, out =65](1,.14)(-1,0) to[in=180+50, out =-65](1,-.14);
}
\,\,\cong\,\,
\tikzmath[scale=.55]{\draw (0,0) circle (1) (-1,0) -- (1,0); \node at (0,-.5) {$\scriptstyle K$};\node at (0,.5) {$\scriptstyle H$};
\node[fill=white, circle, inner sep=-.11] at (-1,.15) {$\scriptscriptstyle \ast$};
\node[fill=white, circle, inner sep=-.11] at (1,-.15) {$\scriptscriptstyle \ast$};
}
\end{equation*}

\noindent
We have the following analogs of Lemma~\ref{lem: property of U} and Lemma~\ref{lem: cocycle property of U}:

\begin{lemma}\label{lem: property of U '}
Let $H\in T^-_\cA$ and $K,L\in T^+_\cA$ be solitons.
Then the following diagram is commutative:
\begin{equation}\label{eq: property of U '}
\tikzmath {
\node(1) at(0,1.4) {$(H\boxtimes K)\boxtimes L$};
\node(2) at(5.5,1.4) {$(H\boxtimes^{\varphi_1} K)\boxtimes L$};
\node(3) at(0,0) {$H\boxtimes (K\boxtimes L)$};
\node(4) at(5.5,0) {$H\boxtimes^{\varphi_1} (K\boxtimes L)$.\!};
\draw[->] (1) --node[above]{$\scriptstyle U_1^{(H,K)}\;\!\boxtimes\;\! \id_L$} (2);
\draw[->] (1) --node[left]{$\scriptstyle a$} (3);
\draw[->] (2) --node[right]{$\scriptstyle a$} (4);
\draw[->] (3) --node[above]{$\scriptstyle U_1^{(H,K\boxtimes L)}$} (4);  %
}
\end{equation}
A similarly diagram holds for $U_2$.
\end{lemma}

\begin{proof}
By Corollary~\ref{cor: NT whose value on the vacuum sector}, the maps in the above diagram only depend on $L$ as an $\cA(I_{34})$-module.
We can therefore proceed as in Lemma~\ref{lem: property of U}.
Let $M:=\{L\in \cA(I_{34})$-Mod $:$ \eqref{eq: property of U '} holds$\}$.
If $L=H_0$, then the horizontal arrows in \eqref{eq: property of U '} can be both identified with $U_1^{(H,K)}$.
So $M$ contains the vacuum sector. % as, in that case, both horizontal arrows in \eqref{eq: property of U '} can be identified with $U_1^{(H,K)}$.
So $M$ is all of $\cA(I_{34})$-Mod.
\end{proof}

\begin{lemma}\label{lem: cocycle property of U '}
Let $H, K\in T^-_\cA$, and $L\in T^+_\cA$ be solitons.
Then the following diagram is commutative:
\begin{equation}\label{eq: cocycle of U '}
\qquad\tikzmath {
\node(1) at(0,1.5) {$H\boxtimes K\boxtimes L$};
\node(2) at(5.5,1.5) {$H\boxtimes^{\varphi_1} (K\boxtimes L)$};
\node(3) at(0,0) {$(H\boxtimes K)\boxtimes^{\varphi_1} L$};
\node(4) at(5.5,0) {$H\boxtimes^{\varphi_1} (K\boxtimes^{\varphi_1} L)$,\!};
\draw[->] (1) --node[above]{$\scriptstyle %U_1^{(H,K)}\;\!\boxtimes\;\! \id_{L}\;=\;
U_1^{(H,K\boxtimes L)}$} (2);
\draw[->] (1) --node[left]{$\scriptstyle U_1^{(H\boxtimes K,L)}$} (3);
\draw[->] (2) --node[right]{$\scriptstyle \id_H\;\!\boxtimes\;\! U_1^{(K,L)}$} (4);
\draw[->] (3) --node[above]{$\scriptstyle a$} (4);  %
}
\end{equation}
where the top horizontal arrow is the one constructed in \eqref{eq: NT defined on A(I_34)-Mod}.
A similarly diagram holds for $U_2$.
\end{lemma}

\begin{proof}
Once again, the maps in \eqref{eq: cocycle of U '} only depend on $L$ as an $\cA(I_{34})$-module.
Let $M:=\{L\in \cA(I_{34})$-Mod $:$ \eqref{eq: cocycle of U '} holds$\}$.
By Lemma~\ref{lem: reconciliation}, we may use the computation \eqref{eq: cocycle property for U_1} to deduce that
$M$ contains the vacuum sector.
%Lemma~\ref{lem: cocycle property of U},
As before, $M$ is closed under taking direct sums and direct summands,
so $M$ is all of $\cA(I_{34})$-Mod.
\end{proof}

Finally, we have:

\begin{proposition}\label{prop: two `hexagon' axioms '}
The braiding isomorphism \eqref{eq: the braiding '} satisfies the two `hexagon' axioms
\[
\beta_{H,K\boxtimes L}=(\id_K\boxtimes \beta_{H,L})(\beta_{H,K}\boxtimes \id_L)
\quad\text{and}\quad
\beta_{H\boxtimes K, L}=(\beta_{H,L}\boxtimes \id_K)(\id_H\boxtimes \beta_{K,L})
\]
(once again, we omit the associators for brevity).
\end{proposition}

\begin{proof}
The proof of Proposition~\ref{prop: two `hexagon' axioms} applies word for word
(use Lemma~\ref{lem: property of U '} in place of Lemma~\ref{lem: property of U},
and Lemma~\ref{lem: cocycle property of U '} in place of Lemma~\ref{lem: cocycle property of U}).
\end{proof}

We will show later, in Proposition~\ref{prop: at most one half-braiding}, that there exists a \emph{unique} braiding
\[
\beta:T^-_\cA\times T^+_\cA\tworarrow\Bim(R)
\]
that satisfies the hexagon axiom $\beta_{H,K\boxtimes L}=(\id_K\boxtimes \beta_{H,L})(\beta_{H,K}\boxtimes \id_L)$.
As a consequence, the braiding \eqref{eq: the braiding '} is independent of the various choices that we made (e.g., the choice of diffeomorphisms $\varphi_1$ and $\varphi_2$).

\subsection{The absorbing object}\label{sec: The absorbing object}

In this section, we recall the results of our earlier paper \cite{Conformal-nets-are-factorization-algebras}, according to which the category of solitons admits an \emph{absorbing object}.
This is the only place where the condition that $\cA$ has finite index is needed.
We start by recalling the definition of an absorbing object:
\begin{definition} %[{\cite[Def.\,5.3]{Bicommutant-categories-from-fusion-categories}}]
An object $\Omega$ of a tensor category $(T,\otimes)$ is called absorbing if it is non-zero and satisfies
\[
\qquad\qquad(X\not= 0) \,\,\Rightarrow\,\, (X\otimes \Omega \,\cong\,\Omega\,\cong\,\Omega\otimes X)\qquad\quad\forall X\in T.
\]
\end{definition}

\noindent If $T$ admits a conjugation (in particular, if $T$ is bi-involutive), then
$\Omega\in T$ is absorbing if and only if it is non-zero and satisfies $X\otimes \Omega \,\cong\,\Omega$ for every $X\not=0$
(see the comments after \cite[Def.\,5.3]{Bicommutant-categories-from-fusion-categories}).

Consider the following manifold (an equilateral triangle):
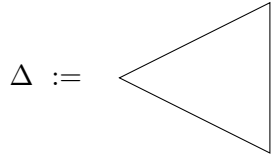
\begin{equation*} %\label{eq: equilateral triangle}
\Delta\,\,:=\quad\tikzmath{\draw (0,0) -- (2,1) -- (2,-1) -- cycle;}
\end{equation*}
equipped with the smooth structure given by constant speed parametrization.
We call the upper left side of this triangle $\Delta_+$, the lower left side $\Delta_-$, and the right side $\Delta_{\mathrm{free}}$.
Let $\Omega:=H_0(\Delta,\cA)$ be the vacuum sector of $\cA$ associated to $\Delta$,
let $S^1_{\mathrm{cut}}$ be as in Section~\ref{sec: Representations and solitons}, and let %us pick a diffeomorphism
\begin{equation*} %\label{eq: S^1_cut --> S_- u S_+}
\varphi_\Delta:\,S^1_{\mathrm{cut}}\,\to\, \Delta_-\cup \Delta_+
\end{equation*}
be the constant speed parametrization
that sends the lower half of $S^1_{\mathrm{cut}}$ to $\Delta_-$ and the upper half of $S^1_{\mathrm{cut}}$ to $\Delta_+$.
We use the diffeomorphism $\varphi_\Delta$ to pull back the action of $\cA(\Delta_-\cup \Delta_+)$ on $\Omega$ to an action of $\cA(S^1_{\mathrm{cut}})$, and thus endow $\Omega$ with the structure of a soliton.
Note that, by Haag duality,
\begin{equation*}
\mathrm{End}_{T_\cA}(\Omega)=\cA(\Delta_{\mathrm{free}}).
\end{equation*}
The following important result was proven in \cite[Thm.\,9]{Conformal-nets-are-factorization-algebras}:

\begin{proposition}\label{prop: Omega is absorbent}
If $\cA$ is a conformal net with finite index, then the object $\Omega\in T_\cA$ is absorbing.
\end{proposition}

\begin{remark}\label{rem: dim le kappa}
It is for the above proposition to hold that it was important to insist that all Hilbert spaces be separable.
If we allow Hilbert spaces of arbitrarily large cardinalities, then the tensor category $T_\cA$ does not have an absorbing object.
\end{remark}

\begin{remark}
In the absence of the finite index condition,
we do not know whether $\Omega$ is absorbing.
\end{remark}

Absorbing objects are important because they control half-braidings:

\begin{proposition}[{\cite[Prop.\,5.9]{Bicommutant-categories-from-fusion-categories}}]
\label{prop:DeltaDeterminesHalfBraidings}
Let $T$ be a category equipped with a tensor functor to $\Bim(R)$.
Let $\Omega\in T$ be an absorbing object, and let $(X,e)$ be in $T'$. Then $e$ is completely determined by its value on $\Omega$.
\end{proposition}
\begin{proof}
Let $Y$ be a non-zero object of $T$.
Since $e$ is a half-braiding, we have a commutative diagram
\[
\tikz{
\node(b) at (2,1.5) {$Y \boxtimes X \boxtimes \Omega$};
\node(a) at (0,0) {$X\boxtimes Y\boxtimes \Omega$};
\node(c) at (4,0) {$Y \boxtimes \Omega \boxtimes X.$};
\draw[->] (a) --node[left, yshift=3]{$\scriptstyle e_{Y}\boxtimes\, \id_\Omega$} (b);
\draw[->] (a) --node[above]{$\scriptstyle e_{Y\boxtimes\Omega}$} (c);
\draw[->] (b) --node[right, yshift=3]{$\scriptstyle \id_Y\boxtimes\, e_{\Omega}$} (c);
}
\]
Pick an isomorphism $\phi:Y\boxtimes \Omega\to \Omega$.
The following square is commutative
\[
\tikz{
\node(a) at (0,1.5) {$X\boxtimes (Y\boxtimes \Omega)$};
\node(b) at (4,1.5) {$(Y\boxtimes \Omega)\boxtimes X$};
\node(c) at (0,0) {$X\boxtimes \Omega$};
\node(d) at (4,0) {$\Omega\boxtimes X$};
\draw[->] (a) --node[above]{$\scriptstyle e_{Y\boxtimes\Omega}$} (b);
\draw[->] (a) --node[left]{$\scriptstyle \id_X \boxtimes\, \phi$} (c);
\draw[->] (b) --node[right]{$\scriptstyle \phi\, \boxtimes \id_X$} (d);
\draw[->] (c) --node[above]{$\scriptstyle e_{\Omega}$} (d);
}
\]
and so we get an equation
\(
e_{Y}\boxtimes \id_\Omega=(\id_Y\boxtimes e_{\Omega}^{-1})\circ(\phi^{-1} \boxtimes \id_X)\circ e_{\Omega}\circ(\id_X \boxtimes \phi)
\).
In particular, we see that $e_{Y}\boxtimes \id_\Omega$ is completely determined by $e_{\Omega}$.
Since $-\boxtimes\Omega$ is a faithful functor,
%$\Bim(R)$ has no zero-divisors,
$e_{Y}$ is completely determined by $e_{Y}\boxtimes \id_\Omega$.
Putting those two facts together, we see that $e_{Y}$ is completely determined by $e_{\Omega}$.
\end{proof}

\subsection{The Drinfel'd center}\label{sec: The Drinfel'd center}

This section is devoted to the proof of our main theorem:

\begin{theorem}\label{thm: main thm 2nd version}
If $\cA$ is a conformal net with finite index.
Then the canonical map $(T^+_\cA)'\to \Bim(R)$ is fully faithful, and $(T^+_\cA)'=T^-_\cA$.
\end{theorem}

The statements in Theorem~\ref{thm: If Rep(cA_G,k) is mod...} are easy consequences of Theorem~\ref{thm: main thm 2nd version}.
The second bullet point in Theorem~\ref{thm: If Rep(cA_G,k) is mod...} is obtained by exchanging the roles of $T^+_\cA$ and $T^-_\cA$, and
the third bullet point is the following computation:
\[
Z(T^+_\cA)=Z_{T^+_\cA}(T^+_\cA)=Z_{\Bim(R)}(T^+_\cA)\cap T^+_\cA=T^-_\cA \cap T^+_\cA=\Rep(\cA),
\]
where we have used Lemma~\ref{lem: essential images} for the last equality.
We note that Lemma~\ref{lem: T --> Bim(R) is fully faithful} (according to which $T^-_\cA$, $T^+_\cA$ and $\Rep(\cA)$ are full subcategories of $\Bim(R)$) has been used implicitly here,
in the usage of the symbol $\cap$.

\begin{proposition}\label{prop: at most one half-braiding}
An object $H\in\Bim(R)$ admits at most one half-braiding with $T_\cA$:
\[
e=(e_K : H \boxtimes K \to K \boxtimes H)_{K\in T_\cA}\;\!.
\]
\end{proposition}

\begin{proof} Let $e$ and $e'$ be half-braidings of $H$ with $T_\cA$.
We wish to show that $e=e'$.
By Proposition~\ref{prop: Omega is absorbent} and Proposition~\ref{prop:DeltaDeterminesHalfBraidings}, it is enough to show that $e_\Omega=e'_\Omega$.
Consider the following $R$-$R$-bimodule map:
\begin{equation}\label{eq: e_Om^-1  circ  e_Om'}
e_\Omega'\circ e_\Omega^{-1}:\,\Omega\boxtimes_R H\,\to\,\Omega\boxtimes_R H.
\end{equation}
By the naturality of $e$ and of $e'$, this map is equivariant for the actions of $\mathrm{End}_{T_\cA}(\Omega)=\cA(\Delta_{\mathrm{free}})$.
We may therefore treat \eqref{eq: e_Om^-1  circ  e_Om'} as a map of $\cA(\Delta_-\cup\Delta_{\mathrm{free}})$-$R$-bimodules.

By Haag duality, $\Omega$ is an invertible $\cA(\Delta_-\cup\Delta_{\mathrm{free}})$-$R$-bimodule \cite[Prop.\,3.10]{BDH(Dualizability+Index-of-subfactors)} (here, $R$ is identified with $\cA(\Delta_+)^\op$).
So there exists an invertible $R$-$R$-bimodule map $u:H\to H$ that satisfies
\begin{equation*} %\label{eq: e'e^-1=u}
e_\Omega'\circ e_\Omega^{-1}=\id_\Omega\boxtimes u.
\end{equation*}
Pick an isomorphism $\omega:\Omega\boxtimes\Omega\to\Omega$ in $T_\cA$ and consider the following diagram:
\[
\tikz{
\node(a) at (-3,1.8) {$H\boxtimes\Omega\boxtimes \Omega$};
\node(b) at (0,3) {$\Omega\boxtimes H\boxtimes\Omega$};
\node(c) at (3,1.8) {$\Omega\boxtimes\Omega\boxtimes H$};
\node(d) at (-2,0) {$H\boxtimes \Omega$};
\node(e) at (2,0) {$\Omega\boxtimes H$};
\draw[->] (a) --node[above, xshift=-8]{$\scriptscriptstyle e_\Omega\boxtimes \id$} (b);
\draw[->] (b) --node[above, xshift=8]{$\scriptscriptstyle \id\boxtimes e_\Omega$} (c);
\draw[->] (a) --node[above]{$\scriptscriptstyle e_{\Omega\boxtimes\Omega}$} (c);
\draw[->] (a) --node[right]{$\scriptscriptstyle \id\boxtimes \omega$} (d);
\draw[->] (d) --node[above]{$\scriptscriptstyle e_{\Omega}$} (e);
\draw[->] (e) --node[left]{$\scriptscriptstyle \omega\boxtimes \id$} (c);
\node(A) at (-5.2,2.5) {$H\boxtimes\Omega\boxtimes \Omega$};
\node(B) at (0,4.7) {$\Omega\boxtimes H\boxtimes\Omega$};
\node(C) at (5.2,2.5) {$\Omega\boxtimes\Omega\boxtimes H$};
\node(D) at (-3,-1.45) {$H\boxtimes \Omega$};
\node(E) at (3,-1.45) {$\Omega\boxtimes H$};
\draw[->] (A) --node[above, xshift=-8]{$\scriptscriptstyle e'_\Omega\boxtimes \id$} (B);
\draw[->] (B) --node[above, xshift=8]{$\scriptscriptstyle \id\boxtimes e'_\Omega$} (C);
%\draw[->] (A) --node[above]{$\scriptscriptstyle e'_{\Omega\boxtimes\Omega}$} (c);
\draw[->] (A) --node[left]{$\scriptscriptstyle \id\boxtimes \omega$} (D);
\draw[->] (D) --node[above]{$\scriptscriptstyle e'_{\Omega}$} (E);
\draw[->] (E) --node[right]{$\scriptscriptstyle \omega\boxtimes \id$} (C);
\draw[->] (a) --node[above, xshift=8]{$\scriptscriptstyle \id$} (A);
\draw[->] (b) to[bend right =20]node[right, yshift=-2]{$\scriptscriptstyle \id$} (B);
\draw[->] (b) to[bend left =20]node[left, yshift=-2]{$\scriptscriptstyle \id\boxtimes u \boxtimes\id$} (B);
\draw[->] (c) --node[above, xshift=-18]{$\scriptscriptstyle \id\boxtimes \id\boxtimes u$} (C);
\draw[->] (d) --node[right, xshift=1]{$\scriptscriptstyle \id$} (D);
\draw[->] (e) --node[left, xshift=-1]{$\scriptscriptstyle \id\boxtimes u$} (E);
}
\]
The middle pentagon commutes because $e$ is a half-braiding.
The outer pentagon commutes by the corresponding property of $e'$. All the quadrilaterals are visibly commutative.
It follows that
\[
\id_\Omega\boxtimes u\boxtimes \id_\Omega=\id_{\Omega\boxtimes H\boxtimes\Omega}.
\]
The functors $\Omega\boxtimes -$ and $-\boxtimes\Omega$ being faithful, we conclude that ${u=\id_H}$.
\end{proof}

\begin{corollary}\label{cor: One beta}
The braiding
\[
\beta:T^-_\cA\times T^+_\cA\tworarrow\Bim(R)
\]
defined in \eqref{eq: the braiding '} is independent of the choices of diffeomorphisms $\varphi_1$ and $\varphi_2$.
The same holds true for its restriction \eqref{eq: the braiding} to $\Rep(\cA)$.
\end{corollary}

\begin{proof}[Proof of Theorem \ref{thm: main thm 2nd version}]
The half-braiding constructed in Section~\ref{sec: the braiding between T^- and T^+} provides a functor
\begin{equation}\label{eq: T^-_A to (T^+_A)'}
T^-_\cA\,\longrightarrow\, (T^+_\cA)'
\end{equation}
that fits in a diagram
\[
\tikz{
\node(a) at (.05,0) {$T^-_\cA$};
\node(b) at (2,.9) {$(T^+_\cA)'$};
\node(c) at (4,0) {$\Bim(R)$};
\draw[->] (a) -- (b);
\draw[->] (a) -- (c);
\draw[->] (b) -- (c);
}
\]
We need to show that the functor \eqref{eq: T^-_A to (T^+_A)'} is an equivalence of categories.
It is clearly faithful as $T^-_\cA\to\Bim(R)$ and $(T^+_\cA)'\to\Bim(R)$ are both faithful functors.
Recall from Lemma~\ref{lem: T --> Bim(R) is fully faithful} that the functor $T^-_\cA\to \Bim(R)$ is fully faithful.
In order to check that the functor \eqref{eq: T^-_A to (T^+_A)'} is an equivalence of categories, it is therefore enough to show that:
\begin{itemize}
\item
The functor $(T^+_\cA)'\to \Bim(R)$ is full (and thus fully faithful).
\item
For every object $Y\in (T^+_\cA)'$, there exists an object $X\in T^-_\cA$ and an isomorphism between their images in $\Bim(R)$.
\end{itemize}
We start by the second item.
By Lemma~\ref{lem: essential images}, it is enough to check that for every object $(H,e) \in (T^+_\cA)'$, 
the actions of $A(I_4)$ and $A(I_1)$ extend to an action of $A(I_{41})$ on $H$.

Recall that $\Omega:=H_0(\Delta,\cA)$, and recall the definition of $\varphi_\Delta:S^1_{\mathrm{cut}}\to \Delta_-\cup \Delta_+$ from the previous section. %and
Let $j:\Delta\to \Delta$ be an orientation reversing involution that satisfies
\[
j(\Delta_-)=\Delta_+\cup \Delta_{\mathrm{free}}
\qquad\quad
j(\Delta_+)=\varphi_\Delta(I_3)
\qquad\quad
j(\Delta_{\mathrm{free}})=\varphi_\Delta(I_4),
\]
and is length-preserving in a neighbourhood of the vertex $\Delta_-\cap \Delta_{\mathrm{free}}$ of $\Delta$.
Recall from \eqref{eq:   v:H_0(S,A) --> L^2(A(I))} that there is a unitary isomorphism
%$\Omega:=H_0(\Delta,\cA)$ is isomorphic to $L^2(\cA(\Delta_-))$ via some map
\[
v:\Omega\to L^2(\cA(\Delta_-))
\]
that intertwines the actions of $\cA(\Delta_-)$, and satisfies $v(\cA(j)(x)\xi)=(v(\xi))x$ for all $x\in \cA(\Delta_-)$ and $\xi\in\Omega$.
Using $\varphi_\Delta$ to identify $I_{34}$ with $\Delta_-$, we get an isomorphism
\[
f:=L^2(\cA(\varphi_\Delta))\circ v:\Omega\to L^2(\cA(\Delta_-))\to L^2(\cA(I_{34}))\cong L^2(R).
\]
Let us write $b:S^1\to S^1$ for the complex conjugation map $z\mapsto\bar z$.
Then the isomorphism $f:\Omega\to L^2(R)$ intertwines the left actions of $\cA(I_{34})$, and satisfies
\[
f\big(\cA(j\circ\varphi_\Delta\circ b)(x)\cdot \xi)=x\cdot(f(\xi)\big)
\]
for all $x\in\cA(I_{12})$ and $\xi\in\Omega$.

Recall that our goal is to show that the actions of $\cA(I_4)$ and $\cA(I_1)$ extend to an action of $\cA(I_{41})$ on $H$.
Consider the isomorphism
\[
e_\Omega:H\boxtimes \Omega\to \Omega\boxtimes H
\]
provided by the half-braiding.
It is a homomorphism of $R$-$R$-bimodules.
By the naturality axiom of half-braidings, it is also equivariant with respect to the actions of $\mathrm{End}_{T^+_\cA}(\Omega)=\cA(\Delta_{\mathrm{free}})$ on $H\boxtimes \Omega$ and on $\Omega\boxtimes H$.
We now consider the composite:
\begin{equation} \label{eq: Om H --> H Om --> H}
F\;\!:\;\,\Omega\boxtimes H\xrightarrow{\,\,e_\Omega^{-1}\,\,} H\boxtimes \Omega\xrightarrow{\,\id\boxtimes f\,}H\boxtimes L^2R\cong H.
\end{equation}
It is equivariant with respect to the left actions of $\cA(I_{34})$, and
intertwines the action of $\cA(\Delta_{\mathrm{free}})$ on $\Omega\boxtimes H$ with the following action on $H$:
\[
\cA(\Delta_{\mathrm{free}})
\xrightarrow{\cA(j\circ\varphi_\Delta\circ b)^{-1}}
\cA(I_1)\to B(H).
\]
We represent the isomorphism \eqref{eq: Om H --> H Om --> H} graphically as follows:
\[
\tikzmath[scale=.55, line join=bevel]{
\useasboundingbox (-1.15,-.8) rectangle (1.1,1.5);
\draw (-1,0) -- (1,.8) -- (1,-.8) -- cycle;
\node at (.3,0) {$\scriptstyle \Omega$};\node at (-.15,.85) {$\scriptstyle H$};
\draw (-1,0) ..controls +(-.5,1.4) and +(-.5,1.3).. (1,.8);
\node[fill=white, circle, inner sep=-.11] at (-1.05,.15) {$\scriptscriptstyle \ast$};
\node[fill=white, circle, inner sep=-.11] at (1-.05,.95) {$\scriptscriptstyle \ast$};
}
\,\,\xrightarrow{\,e_\Omega^{-1}\,}\,\,
\tikzmath[scale=.55, yscale=-1, line join=bevel]{
\useasboundingbox (-1.15,-.8) rectangle (1.1,1.5);
\draw (-1,0) -- (1,.8) -- (1,-.8) -- cycle;
\node at (.3,0) {$\scriptstyle \Omega$};\node at (-.15,.85) {$\scriptstyle H$};
\draw (-1,0) ..controls +(-.5,1.4) and +(-.5,1.3).. (1,.8);
\node[fill=white, circle, inner sep=-.11] at (-1.05,.15) {$\scriptscriptstyle \ast$};
\node[fill=white, circle, inner sep=-.11] at (1-.05,.95) {$\scriptscriptstyle \ast$};
}
\,\,\xrightarrow{\id\boxtimes f}\,\,
\tikzmath[scale=.65]{\draw (0,0) circle (1) (-1,0) -- (1,0); \node at (0,-.45) {$\scriptstyle H$};\node at (0,.45) {$\scriptstyle H_0$};
\node[fill=white, circle, inner sep=0] at (-1,-.135) {$\scriptscriptstyle \ast$};
\node[fill=white, circle, inner sep=0] at (1,-.135) {$\scriptscriptstyle \ast$};
}
\,\,\cong\,\,
\tikzmath[scale=.65]{\draw (0,0) circle (1); \node at (0,0) {$H$};
\node[fill=white, circle, inner sep=0] at (-1,0) {$\scriptscriptstyle \ast$};
\node[fill=white, circle, inner sep=0] at (1,0) {$\scriptscriptstyle \ast$};
}
\]
The little stars are there to indicate that $H$ is a priori a mere bimodule, as opposed to a soliton or a representation.

Let us write
\[
\rho^H_i:\cA(I_i)\to B(H)\qquad\text{and}\qquad \rho^\Omega_i:\cA(I_i)\to B(\Omega)
\]
for the actions of $\cA(I_i)$ on $H$ and on $\Omega$, and let us write $\alpha$ for the following action of $\cA(I_1)$ on $\Omega$:
\[
\alpha:\cA(I_1)
\xrightarrow{\,\,\cA(j\circ\varphi_\Delta\circ b)\,\,}
\cA(\Delta_{\mathrm{free}})
\to B(\Omega).
\]
By construction, the map \eqref{eq: Om H --> H Om --> H} satisfies
\begin{equation}\label{eq: intertwining property of F}
\begin{split}
F\circ \big(\rho_4^\Omega (x)\boxtimes \id_H\big) = \rho_4^H (x) \circ F\qquad\quad \forall x\in \cA(I_4)\;
\\
F\circ \big(\alpha (x)\boxtimes \id_H\big) = \rho_1^H (x) \circ F\qquad\quad\; \forall x\in \cA(I_1).
\end{split}
\end{equation}
Since $j$ is an isometry in a neighbourhood of $\Delta_-\cap \Delta_{\mathrm{free}}$,
the maps $\varphi_\Delta:I_4\to \Delta$ and $j\circ\varphi_\Delta\circ b:I_1\to \Delta$ extend to a smooth map
\[
\varphi_\Delta\cup (j\circ\varphi_\Delta\circ b): I_{41}\to \Delta.
\]
The actions $\rho_4^\Omega:\cA(I_4)\to B(\Omega)$ and $\alpha:\cA(I_1)\to B(\Omega)$ therefore extend to an action of $A(I_{14})$ on $\Omega$.
By the intertwining properties \eqref{eq: intertwining property of F} of $F$, 
the actions $\rho_4^H:\cA(I_4)\to B(H)$ and $\rho_1^H:\cA(I_1)\to B(H)$ therefore also extend to an action of $A(I_{14})$ on $H$.
This finishes the proof of the second item in the bullet list.

We now turn our attention to the first item in the list.
Let us write
\[
s:T^-_\cA\to (T^+_\cA)'
\] 
for the functor \eqref{eq: T^-_A to (T^+_A)'}. % constructed in Section~\ref{sec: the braiding between T^- and T^+}.
Let $(H_1,e_1)$ and $(H_2,e_2)$ be objects of $(T^+_\cA)'$, and let $f:H_1\to H_2$ be a morphism between their images in $\Bim(R)$.
In the first half of the proof, we learned that $H_1$ and $H_2$ are in fact objects of $T^-_\cA$.
By Lemma~\ref{lem: T --> Bim(R) is fully faithful}, $f$ is a morphism in $T^-_\cA$.
By Proposition~\ref{prop: at most one half-braiding}, $s(H_1)=(H_1,e_1)$ and $s(H_2)=(H_2,e_2)$.
Therefore,
\[
s(f):(H_1,e_1)\to (H_2,e_2)
\]
is a morphism in $(T^+_\cA)'$.
Now, by construction, $s(f)$ maps to $f$ under the forgetful map $(T^+_\cA)'\to \Bim(R)$.
\end{proof}

\begin{remark}\label{rem: Z versus dot Z}
The arguments in the proofs of Proposition~\ref{prop: at most one half-braiding} and of Theorem \ref{thm: main thm 2nd version} never used the fact that the half-braidings are unitary
(it just so happens that every half-braiding with $T^+_\cA$ is unitary).
The non-unitary version of Theorem~\ref{thm: If Rep(cA_G,k) is mod...} therefore also holds:
\[
\dot Z_{\Bim(R)}(T^+_\cA)\cong T^-_\cA,\qquad
\dot Z_{\Bim(R)}(T^-_\cA)\cong T^+_\cA,\qquad
\dot Z(T^+_\cA)\cong\Rep(\cA).
\]
\end{remark}

\bibliographystyle{amsalpha}
\bibliography{FromConfNetstoBicomCat}

\newcommand{\etalchar}[1]{$^{#1}$}
\def\cprime{$'$} \def\cprime{$'$} \def\cprime{$'$} \def\cprime{$'$}
\providecommand{\bysame}{\leavevmode\hbox to3em{\hrulefill}\thinspace}
\providecommand{\MR}{\relax\ifhmode\unskip\space\fi MR }
\providecommand{\MRhref}[2]{%
  \href{http://www.ams.org/mathscinet-getitem?mr=#1}{#2}
}
\providecommand{\href}[2]{#2}
\begin{thebibliography}{DMNO13}

\bibitem[BD95]{MR1355899}
John~C. Baez and James Dolan, \emph{Higher-dimensional algebra and topological
  quantum field theory}, J. Math. Phys. \textbf{36} (1995), no.~11, 6073--6105.
  \MR{1355899 (97f:18003)}

\bibitem[BDH14]{BDH(Dualizability+Index-of-subfactors)}
Arthur Bartels, Christopher~L. Douglas, and Andr{\'e} Henriques,
  \emph{Dualizability and index of subfactors}, Quantum Topol. \textbf{5}
  (2014), no.~3, 289--345. \MR{3342166}

\bibitem[BDH15]{BDH-cn1}
\bysame, \emph{Conformal nets {I}: {C}oordinate-free nets}, Int. Math. Res.
  Not. IMRN (2015), no.~13, 4975--5052. \MR{3439097}

\bibitem[BDH16]{BDH-cn3}
\bysame, \emph{Fusion of defects (formerly conformal nets {III}: fusion of
  defects)}, Memoirs of the AMS (2016).

\bibitem[BDH17]{BDH-cn2}
Arthur Bartels, Christopher~L. Douglas, and Andr\'e Henriques, \emph{Conformal
  {N}ets {II}: {C}onformal {B}locks}, Comm. Math. Phys. \textbf{354} (2017),
  no.~1, 393--458. \MR{3656522}

\bibitem[BE98]{MR1652746}
J.~B{\"o}ckenhauer and D.~E. Evans, \emph{Modular invariants, graphs and
  {$\alpha$}-induction for nets of subfactors. {I}}, Comm. Math. Phys.
  \textbf{197} (1998), no.~2, 361--386. \MR{1652746 (2000c:46121)}

\bibitem[BK01]{MR1797619}
Bojko Bakalov and Alexander Kirillov, Jr., \emph{Lectures on tensor categories
  and modular functors}, University Lecture Series, vol.~21, American
  Mathematical Society, Providence, RI, 2001. \MR{1797619 (2002d:18003)}

\bibitem[BSM90]{MR1079298}
Detlev Buchholz and Hanns Schulz-Mirbach, \emph{Haag duality in conformal
  quantum field theory}, Rev. Math. Phys. \textbf{2} (1990), no.~1, 105--125.
  \MR{1079298 (92a:81106)}

\bibitem[BW96]{MR1357878}
John~W. Barrett and Bruce~W. Westbury, \emph{Invariants of piecewise-linear
  {$3$}-manifolds}, Trans. Amer. Math. Soc. \textbf{348} (1996), no.~10,
  3997--4022. \MR{1357878}

\bibitem[CJM{\etalchar{+}}05]{MR2174418}
Alan~L. Carey, Stuart Johnson, Michael~K. Murray, Danny Stevenson, and Bai-Ling
  Wang, \emph{Bundle gerbes for {C}hern-{S}imons and {W}ess-{Z}umino-{W}itten
  theories}, Comm. Math. Phys. \textbf{259} (2005), no.~3, 577--613.
  \MR{2174418 (2007a:58023)}

\bibitem[CKLW15]{arXiv:1503.01260}
Sebastiano Carpi, Yasuyuki Kawahigashi, Roberto Longo, and Mih{\' a}ly Weiner,
  \emph{From vertex operator algebras to conformal nets and back},
  arXiv:1503.01260, to appear in Memoirs of the AMS, 2015.

\bibitem[Con94]{MR1303779}
Alain Connes, \emph{Noncommutative geometry}, Academic Press, Inc., San Diego,
  CA, 1994. \MR{1303779 (95j:46063)}

\bibitem[DLM96]{MR1408523}
Chongying Dong, Haisheng Li, and Geoffrey Mason, \emph{Simple currents and
  extensions of vertex operator algebras}, Comm. Math. Phys. \textbf{180}
  (1996), no.~3, 671--707. \MR{1408523 (97j:17031)}

\bibitem[DMNO13]{MR3039775}
Alexei Davydov, Michael M{\"u}ger, Dmitri Nikshych, and Victor Ostrik,
  \emph{The {W}itt group of non-degenerate braided fusion categories}, J. Reine
  Angew. Math. \textbf{677} (2013), 135--177. \MR{3039775}

\bibitem[DSPS13]{arXiv:1312.7188}
Christopher~L. Douglas, Christopher Schommer-Pries, and Noah Snyder,
  \emph{Dualizable tensor categories}, arXiv:1312.7188, 2013.

\bibitem[DW90]{MR1048699}
Robbert Dijkgraaf and Edward Witten, \emph{Topological gauge theories and group
  cohomology}, Comm. Math. Phys. \textbf{129} (1990), no.~2, 393--429.
  \MR{1048699 (91g:81133)}

\bibitem[ENO05]{MR2183279}
Pavel Etingof, Dmitri Nikshych, and Viktor Ostrik, \emph{On fusion categories},
  Ann. of Math. (2) \textbf{162} (2005), no.~2, 581--642. \MR{2183279}

\bibitem[FGK88]{MR946997}
G.~Felder, K.~Gawedzki, and A.~Kupiainen, \emph{Spectra of
  {W}ess-{Z}umino-{W}itten models with arbitrary simple groups}, Comm. Math.
  Phys. \textbf{117} (1988), no.~1, 127--158. \MR{946997 (89k:81077)}

\bibitem[FHLT10]{MR2648901}
Daniel~S. Freed, Michael~J. Hopkins, Jacob Lurie, and Constantin Teleman,
  \emph{Topological quantum field theories from compact {L}ie groups}, A
  celebration of the mathematical legacy of {R}aoul {B}ott, CRM Proc. Lecture
  Notes, vol.~50, Amer. Math. Soc., Providence, RI, 2010, pp.~367--403.
  \MR{2648901 (2011i:57040)}

\bibitem[Fr{\"o}76]{MR0413868}
J{\"u}rg Fr{\"o}hlich, \emph{New super-selection sectors (``soliton-states'')
  in two dimensional {B}ose quantum field models}, Comm. Math. Phys.
  \textbf{47} (1976), no.~3, 269--310. \MR{0413868 (54 \#1980)}

\bibitem[FRS89]{MR1016869}
K.~Fredenhagen, K.-H. Rehren, and B.~Schroer, \emph{Superselection sectors with
  braid group statistics and exchange algebras. {I}.\ {G}eneral theory}, Comm.
  Math. Phys. \textbf{125} (1989), no.~2, 201--226. \MR{1016869 (91c:81047)}

\bibitem[FSS96]{MR1409292}
J.~Fuchs, A.~N. Schellekens, and C.~Schweigert, \emph{A matrix {$S$} for all
  simple current extensions}, Nuclear Phys. B \textbf{473} (1996), no.~1-2,
  323--366. \MR{1409292 (97i:81041)}

\bibitem[FSS15]{MR3330242}
Domenico Fiorenza, Hisham Sati, and Urs Schreiber, \emph{A higher stacky
  perspective on {C}hern-{S}imons theory}, Mathematical aspects of quantum
  field theories, Math. Phys. Stud., Springer, Cham, 2015, pp.~153--211.
  \MR{3330242}

\bibitem[GF93]{MR1231644}
Fabrizio Gabbiani and J{\"u}rg Fr{\"o}hlich, \emph{Operator algebras and
  conformal field theory}, Comm. Math. Phys. \textbf{155} (1993), no.~3,
  569--640. \MR{1231644 (94m:81090)}

\bibitem[Gui66]{MR0201989}
Alain Guichardet, \emph{Sur la cat\'egorie des alg\`ebres de von {N}eumann},
  Bull. Sci. Math. (2) \textbf{90} (1966), 41--64. \MR{0201989}

\bibitem[Haa75]{MR0407615}
Uffe Haagerup, \emph{The standard form of von {N}eumann algebras}, Math. Scand.
  \textbf{37} (1975), no.~2, 271--283. \MR{0407615 (53 \#11387)}

\bibitem[Hen15]{CS(pt)}
Andr{\'e} Henriques, \emph{What {C}hern--{S}imons theory assigns to a point},
  arXiv:1503.06254, 2015.

\bibitem[Hen16]{WZW-classification}
\bysame, \emph{The classification of chiral {WZW} models by
  ${H}^4_+({B}{G},\mathbb {Z})$}, Contemporary Mathematics (2016).

\bibitem[Hen17a]{Conformal-nets-are-factorization-algebras}
\bysame, \emph{Conformal nets are factorization algebras}, Proceedings of the
  String-Math 2016 conference (2017).

\bibitem[Hen17b]{colimits}
\bysame, \emph{Loop groups and diffeomorphism groups of the circle as
  colimits}, arXiv:1706.08471, 2017.

\bibitem[HP17]{Bicommutant-categories-from-fusion-categories}
Andr\'e Henriques and David Penneys, \emph{Bicommutant categories from fusion
  categories}, Selecta Math. (N.S.) \textbf{23} (2017), no.~3, 1669--1708.
  \MR{3663592}

\bibitem[JS91]{MR1107651}
Andr{\'e} Joyal and Ross Street, \emph{Tortile {Y}ang-{B}axter operators in
  tensor categories}, J. Pure Appl. Algebra \textbf{71} (1991), no.~1, 43--51.
  \MR{1107651 (92e:18006)}

\bibitem[Kaw02]{MR1892455}
Yasuyuki Kawahigashi, \emph{Generalized {L}ongo-{R}ehren subfactors and
  {$\alpha$}-induction}, Comm. Math. Phys. \textbf{226} (2002), no.~2,
  269--287. \MR{1892455 (2003i:46067)}

\bibitem[KLM01]{MR1838752}
Yasuyuki Kawahigashi, Roberto Longo, and Michael M{\"u}ger,
  \emph{Multi-interval subfactors and modularity of representations in
  conformal field theory}, Comm. Math. Phys. \textbf{219} (2001), no.~3,
  631--669. \MR{1838752 (2002g:81059)}

\bibitem[Li01]{MR1822111}
Haisheng Li, \emph{Certain extensions of vertex operator algebras of affine
  type}, Comm. Math. Phys. \textbf{217} (2001), no.~3, 653--696. \MR{1822111
  (2003a:17036)}

\bibitem[Lon89]{MR1027496}
Roberto Longo, \emph{Index of subfactors and statistics of quantum fields.
  {I}}, Comm. Math. Phys. \textbf{126} (1989), no.~2, 217--247. \MR{1027496}

\bibitem[Lon08]{Longo(Lectures-on-Nets-II)}
R.~Longo, \emph{Lectures on conformal nets {II} \hfill},
  http:/\!/www.mat.uniroma2.it/$\sim${}longo/Lecture\% 20Notes.html, 2008.

\bibitem[LR95]{MR1332979}
R.~Longo and K.-H. Rehren, \emph{Nets of subfactors}, Rev. Math. Phys.
  \textbf{7} (1995), no.~4, 567--597, Workshop on Algebraic Quantum Field
  Theory and Jones Theory (Berlin, 1994). \MR{1332979 (96g:81151)}

\bibitem[Lur09]{MR2555928}
Jacob Lurie, \emph{On the classification of topological field theories},
  Current developments in mathematics, 2008, Int. Press, Somerville, MA, 2009,
  pp.~129--280. \MR{2555928 (2010k:57064)}

\bibitem[Lur11]{Lurie(Lectures-on-vnAlg)}
J.~Lurie, \emph{Lectures notes on von {N}eumann algebras \hfill},
  http:/\!/www.math.harvard.edu/$\sim${}lurie/261y.html, 2011.

\bibitem[LX04]{MR2100058}
Roberto Longo and Feng Xu, \emph{Topological sectors and a dichotomy in
  conformal field theory}, Comm. Math. Phys. \textbf{251} (2004), no.~2,
  321--364. \MR{2100058 (2005i:81087)}

\bibitem[Maj91]{MR1151906}
Shahn Majid, \emph{Representations, duals and quantum doubles of monoidal
  categories}, Proceedings of the {W}inter {S}chool on {G}eometry and {P}hysics
  ({S}rn\'\i, 1990), no.~26, 1991, pp.~197--206. \MR{1151906}

\bibitem[MS89]{MR992362}
Gregory Moore and Nathan Seiberg, \emph{Taming the conformal zoo}, Phys. Lett.
  B \textbf{220} (1989), no.~3, 422--430. \MR{992362 (90e:81217)}

\bibitem[MTW16]{arXiv:1609.02169}
V.~Morinelli, Y~Tanimoto, and M.~Weiner, \emph{Conformal covariance and the
  split property}, arXiv:1609.02169, 2016.

\bibitem[M{\"u}g03]{MR1966525}
Michael M{\"u}ger, \emph{From subfactors to categories and topology. {II}.
  {T}he quantum double of tensor categories and subfactors}, J. Pure Appl.
  Algebra \textbf{180} (2003), no.~1-2, 159--219. \MR{1966525 (2004f:18014)}

\bibitem[RT91]{MR1091619}
N.~Reshetikhin and V.~G. Turaev, \emph{Invariants of {$3$}-manifolds via link
  polynomials and quantum groups}, Invent. Math. \textbf{103} (1991), no.~3,
  547--597. \MR{1091619}

\bibitem[Sau83]{MR703809}
Jean-Luc Sauvageot, \emph{Sur le produit tensoriel relatif d'espaces de
  {H}ilbert}, J. Operator Theory \textbf{9} (1983), no.~2, 237--252. \MR{703809
  (85a:46034)}

\bibitem[Tak03]{MR1943006}
M.~Takesaki, \emph{Theory of operator algebras. {II}}, Encyclopaedia of
  Mathematical Sciences, vol. 125, Springer-Verlag, Berlin, 2003, Operator
  Algebras and Non-commutative Geometry, 6. \MR{1943006 (2004g:46079)}

\bibitem[Tho11]{MR2771095}
Andreas Thom, \emph{A remark about the {C}onnes fusion tensor product}, Theory
  Appl. Categ. \textbf{25} (2011), No. 2, 38--50. \MR{2771095 (2012g:46085)}

\bibitem[TV92]{MR1191386}
V.~G. Turaev and O.~Ya. Viro, \emph{State sum invariants of {$3$}-manifolds and
  quantum {$6j$}-symbols}, Topology \textbf{31} (1992), no.~4, 865--902.
  \MR{1191386}

\bibitem[Was98]{MR1645078}
Antony Wassermann, \emph{Operator algebras and conformal field theory. {III}.
  {F}usion of positive energy representations of {${\rm LSU}(N)$} using bounded
  operators}, Invent. Math. \textbf{133} (1998), no.~3, 467--538. \MR{1645078
  (99j:81101)}

\bibitem[Wit89]{MR990772}
Edward Witten, \emph{Quantum field theory and the {J}ones polynomial}, Comm.
  Math. Phys. \textbf{121} (1989), no.~3, 351--399. \MR{990772 (90h:57009)}

\bibitem[Wra10]{Wray-thesis}
Kevin Wray, \emph{Extended topological gauge theories in codimension zero and
  higher}, Master's Thesis, Universiteit van Amsterdam, 2010.

\bibitem[Xu00]{MR1776984}
Feng Xu, \emph{Jones-{W}assermann subfactors for disconnected intervals},
  Commun. Contemp. Math. \textbf{2} (2000), no.~3, 307--347. \MR{1776984
  (2001f:46094)}

\bibitem[Yam92]{MR1203761}
Shigeru Yamagami, \emph{Algebraic aspects in modular theory}, Publ. Res. Inst.
  Math. Sci. \textbf{28} (1992), no.~6, 1075--1106. \MR{1203761 (94c:46123)}

\end{thebibliography}

\end{document}